\documentclass[11pt]{amsart}
\usepackage{cases}
\usepackage{latexsym}
\usepackage{amsmath}
\usepackage[arrow,matrix]{xy}
\usepackage{stmaryrd}
\usepackage{amsfonts}
\usepackage{amsmath, amssymb,amscd,amsthm, amsxtra, bbm, mathrsfs}
\usepackage{cite}
\usepackage{array}
\usepackage[utf8]{inputenc}
\usepackage{fancyhdr}
\usepackage{ifthen}
\usepackage{verbatim}
\usepackage{graphics}
\usepackage{tikz}
\usepackage{tikz-cd}
\usepackage[a4paper, top=30mm, bottom=35mm, left=25mm, right=25mm]{geometry}

\theoremstyle{plain}
\newtheorem{thm}{Theorem}[section]
\newtheorem{cor}[thm]{Corollary}
\newtheorem{lem}[thm]{Lemma}
\newtheorem{prop}[thm]{Proposition}

\theoremstyle{definition}

\numberwithin{figure}{section}
\newtheorem{rem}{Remark}

\theoremstyle{remark}

\numberwithin{equation}{section}

\usepackage{enumitem}
\setlist[enumerate]{left=0pt}

\begin{document}
	
	\thispagestyle{empty}
	
	\title[Embedding of line graph associated to the essential graph of rings]{Embeddings of the line graphs associated with the essential graphs of commutative rings}
	\author{Sakshi Jain, Mohd Nazim, Y. M. Borse}

	\address{Sakshi Jain, Department of Mathematics, Savitribai Phule Pune University, Pune-411007, India }
	\email{sakshi99chordiya@gmail.com}
	
	\address{Mohd Nazim, School of Computational Sciences, Faculty of Science and Technology, JSPM University, Pune-412207, India}
	\email{mnazim1882@gmail.com}

	\address{Y. M. Borse, Department of Mathematics, Savitribai Phule Pune University, Pune-411007, India}
	\email{ymborse11@gmail.com}

	\begin{abstract}
		Let $A$ be a finite commutative ring with unity $1 \neq 0.$  An ideal of \( A \) is said to be \emph{essential} if it has a non-zero intersection with every non-zero ideal of \( A \). The \emph{essential graph} of \( A \) is a simple undirected graph whose vertex set consists of all non-zero zero-divisors of \( A \). Two different vertices \( u \) and \( v \) are connected by an edge precisely when the ideal formed by the annihilator of their product \( uv \) is essential in \( A \).
		This paper examines the minimal embeddings of the line graph of the essential graph of $A$ into orientable surfaces as well as non-orientable surfaces. Our results include a complete classification of finite commutative rings for which the line graphs of their essential graphs is planar, outerplanar or have genus or crosscap number at most two. We also characterize all such non-local rings for which the line graph of their zero-divisor graph is outerplanar.
		
	\end{abstract}
	\maketitle    
	
	
	\section{Introduction}
	
	\noindent In algebraic graph theory, the creation of graphs from rings is considered an important construction. Analyzing graphs that arise from rings offers valuable insights into the connection between the algebraic structures of the rings and the combinatorial properties of the resulting graphs. This direction of research began with the concept of zero-divisor graphs of commutative rings, first introduced by Beck~\cite{beck} in 1988. Anderson and Livingston \cite{anderson-livingston} later modified this concept in 1999. In the \textit{zero-divisor graph} \( \Gamma(A) \) of a commutative ring \( A \) with identity, the vertices represent the non-zero zero-divisors of \( A \), and two distinct vertices are connected by an edge if their product in \( A \) is zero.  Over the years, various researchers have examined these graphs from structural, spectral, and topological perspectives; see ~\cite{sonawane, hjwang, Asir}.	
	Several generalizations of the zero-divisor graphs with nearly the same vertex set have been introduced, including the annihilator graph\cite{badawi}, cozero-divisor graph\cite{afkhamicozerodivisor} and weakly zero-divisor graph\cite{weaklyzerodivisor}. Among these, Nikmehr et al.~\cite{Nik1} introduced one such generalization, known as the \textit{essential graph} in 2017. Denoted by \( EG(A) \), the essential graph of a commutative ring \( A \) has as its vertex set the non-zero zero-divisors of \( A \), where two distinct vertices are adjacent if the annihilator of their product is an essential ideal in the ring $A.$  As a result, the zero-divisor graph \( \Gamma(A) \) naturally appears as a subgraph of \( EG(A) \). Nikmehr et al.~\cite{Nik1} further established that the essential graph is always connected, with diameter at most $3$ and girth at most $4$ whenever it contains a cycle. Subsequent studies have focused on both the structural and topological characteristics of essential graph; see~\cite{Nik1, selvakumar-2, selvakumar-subajini, kalaimurugan}. In particular, graphs such as trees, unicyclic graphs, split graphs, and outerplanar graphs were classified by Selvakumar et al.~\cite{selvakumar-subajini} in terms of the finite commutative rings \( A \) for which the essential graph \( EG(A) \) exhibits these structures. They also classified the rings for which \( EG(A) \) is planar, toroidal, or embeddable in the projective plane.  Moreover, Selvakumar et al.~\cite{selvakumar-2} characterized the rings whose essential graphs have genus two. In a more recent study, Kalaimurugan et al.~\cite{kalaimurugan} examined the interplay among zero-divisor graphs, essential graphs, and nilpotent graphs, extending genus-two results from essential graphs to their nilpotent counterparts. Various extensions of this graph are introduced and studied; see \cite{Nazim, Amjadi}
	
	\par A fundamental result in graph theory asserts that every connected graph except $K_3$ and $K_{1,3}$ is uniquely determined by its line graph, establishing a one-to-one mapping between connected graphs and their line graphs.
	Expanding upon this direction, Barati et al.~\cite{barati} classified all finite commutative rings whose zero-divisor graphs are either line graphs or complements of line graphs.
	A complete characterization of finite commutative rings yielding essential graphs that are line graphs or complements of line graphs was established in the work of Rehman et al.~\cite{rehman}. They also studied a similar classification for weakly zero-divisor graphs. More broadly, characterizing line graphs arising from algebraic objects such as groups and rings has remained a central topic of interest in the literature; for further developments in this direction, see~\cite{bera}.
	
	\par  The study of line graph embeddings has been the subject of several notable investigations across different algebraically defined graphs. A foundational result in this direction was provided by Bénard~\cite{MR485482}, who established a lower bound for the genus of a line graph of a complete graph and analyzed its behavior across specific graph classes. Later, the work of Chiang-Hsieh et al.~\cite{chi-hsi} focused on zero-divisor graphs, offering a classification of all finite commutative rings whose line graphs exhibit genus and crosscap number at most two.
	Similarly, the genus of line graphs arising from co-maximal graphs was studied by Huadong et al.~\cite{Huadongcomaximal}, with an emphasis on embeddings into surfaces up to genus two.  More recently, a comprehensive study by Shariq et al.~\cite{mshariq} examined the embeddings of the line graphs associated with annihilator graphs. They classified all finite commutative rings for which these line graphs are planar, outerplanar, or possess genus and crosscap number no greater than two.
	Inspired by the previous studies, we aim to investigate the embeddings on orientable and non-orientable surfaces of the line graphs associated with the essential graphs. 
	
	In this paper, we establish a complete characterization of the finite commutative rings that corresponds to planar or outerplanar line graphs of the essential graphs. Moreover, we characterize the finite commutative rings for which the genus or crosscap number of the line graphs of the essential graphs is at most two. As an additional result, we investigate non-local rings whose zero-divisor graphs have line graphs that are outerplanar. These results extend the classifications known for the line graphs of zero-divisor graphs to the broader framework of essential graphs. Section 2 outlines the necessary preliminaries, Section 3 focuses on planarity and outerplanarity, and Section 4 and 5 addresses the genus and crosscap properties, respectively.  
	
	\section{Preliminaries}
	
	
	\noindent In this section, we outline some basic preliminary concepts of graphs and rings. We begin with graphs. Let \(G\) be an undirected simple graph with vertex set and edge set denoted by \(V(G)\) and \(E(G)\), respectively. For any vertex \( v \in V(G) \), the degree of \( v \), written as \( \deg(v) \), refers to the total number of edges incident to it. A graph is called \emph{complete} if every pair of distinct vertices is connected by exactly one edge; such a graph on \( m \) vertices is denoted by \( K_m \). A \emph{complete bipartite graph} is a graph \( G \) such that its vertex set can be divided into two non-overlapping subsets \( S_1 \) and \( S_2 \) such that each vertex in \( S_1 \) is adjacent to every vertex in \( S_2 \), and there are no edges connecting vertices within the same subset. This graph is represented by \( K_{n,m} \), where \( n = |S_1| \) and \( m = |S_2| \). Extending this idea, a \emph{complete tripartite graph}, denoted by \( K_{m,n,r} \), consists of three mutually disjoint vertex sets of sizes \( m \), \( n \), and \( r \), with edges connecting every vertex in one subset to all vertices in the other two, and no edge connecting vertices within the same subset.  For a graph $G$, the vertex set of a \emph{line graph} $L(G)$ is the set of edges of $G,$ and two vertex are adjacent precisely when the corresponding edges in \( G \) share a common vertex.
	
	\par All rings considered in this paper are finite, commutative and with unity $1 \neq 0.$ Let $A$ be one of such rings. A ring \( A \) is termed \emph{local} if it contains exactly one maximal ideal \( \mathfrak{m}.\) If \( A \) is not local, then by the Artinian ring structure theorem~\cite{atiyah-macdonald}, \( A \) decomposes as a finite direct product of local rings. Moreover, \( A \) is \emph{reduced} precisely when it is isomorphic to a product \( F_1 \times \cdots \times F_k \), where each \( F_i \) is a finite field and \( k \in \mathbb{N} \). For any \( s \in A \), its \emph{annihilator}, denoted \( \mathrm{ann}_A(s) \), is defined as \( \{ r \in A \mid sr = 0 \} \). An ideal \( I \subseteq A \) is said to be \emph{essential} if it intersects non-trivially with every non-zero ideal of \( A \). Note that \( \mathrm{ann}_A(0) = A \), and that \( A \) itself is an essential ideal. The set of zero-divisors is written as \( Z(A) \), and the group of units is denoted by \( A^\times \).
	For  \( S \subseteq A\), define \( S^* = S \setminus \{ 0 \}. \)  For an integer \( n \geq 2 \), let \( \mathbb{Z}_n \) denote the ring of integers modulo \( n \) and \( \mathbb{F}_n \) denote the finite field with \( n \) elements. An element \( s \in A \) is called \emph{nilpotent} if there exists some integer \( k >0 \) such that \( s^k = 0 \). In a local ring, every non-zero element is either a unit or nilpotent. For notation or terminology related to ring theory that is not introduced here, we refer~\cite{atiyah-macdonald}. 
	
	The following results highlight important structural aspects of the essential graph.
	
	\begin{lem}\label{complete bipartite}
		For finite fields $\mathbb{F}_n$ and $ \mathbb{F}_m,$ $EG(\mathbb{F}_n\times \mathbb{F}_m) \cong K_{n -1, m -1}.$
	\end{lem} 
	\begin{proof}
		Since $ \mathbb{F}_n\times \mathbb{F}_m = \{ (u, v) \colon \, u \in \mathbb{F}_n, ~v \in \mathbb{F}_m\},$ then its essential graph is a complete bipartite graph with bipartition $ X = \{ (u,0) \colon \, u \in \mathbb{F}_n^* \} $ and $ X' = \{ (0,v) \colon \, v \in \mathbb{F}_m^* \} $ of size $n -1$ and $m -1$, respectively.
	\end{proof}
	
	\begin{lem}\label{Zero divisor graph = essential graph}\textnormal{\cite{Nik1}}
		For reduced ring \( A \), it follows that \( EG(A) = \Gamma(A) \).
	\end{lem}
	
	\begin{lem}\label{cg}\textnormal{\cite{Nik1}}
		Suppose \( A \) is a non-reduced ring. Then:
		\begin{enumerate}[label=\textnormal{(\roman*)}]
			
			\item Each non-zero nilpotent element in \( A \) is connected to every vertex in \( EG(A) \).
			
			\item The collection of all non-zero nilpotent elements induces a clique in \( EG(A) \).
		\end{enumerate}

	\end{lem}

	\begin{cor}\label{fcg} For local ring \( A \), \( EG(A) \) is a complete graph.
	\end{cor}
	
	\begin{proof}
		Since $A$ is a local ring, every non-zero zero-divisor is nilpotent. Now by Lemma \ref{cg} the result follows. 
	\end{proof}
	
	The following table lists all local rings \( A \) whose essential graphs \( EG(A) \) contain upto $6$ vertices. Observe that no local ring possesses precisely five nonzero zero-divisors.
	

	\begin{table}[h!]
		\centering
		
		\begin{tabular} { |  >{\centering\arraybackslash}p {2 cm} | >{\centering\arraybackslash}p {7 cm} | >{\centering\arraybackslash}p {3 cm} | }
			\hline
			$|Z(A)^*|$ &
			$A$ - Local ring & 		
			Graph $EG(A)$ \\
			\hline
			\rule{0pt}{15pt} 1 & $\mathbb{Z}_4$, $\frac{\mathbb{Z}_2[x]}{\left\langle x^2 \right\rangle}$   & $K_1$ \\
			\hline
			\rule{0pt}{15pt} 2 & $\mathbb{Z}_9$, $\frac{\mathbb{Z}_3[x]}{\left\langle x^2 \right\rangle}$ & $K_2$ \\
			\hline
			\rule{0pt}{20pt} 3 &
			$\mathbb{Z}_8$,
			$\frac{\mathbb{Z}_2[x]}{\left\langle x^3 \right\rangle}$, $\frac{\mathbb{Z}_4[x]}{\left\langle x^3, x^2 - 2 \right\rangle}$, $\frac{\mathbb{Z}_4[x]}{\left\langle 2x, x^2 \right\rangle}$, $\frac{\mathbb{F}_4[x]}{\left\langle x^2 \right\rangle}$, $\frac{\mathbb{Z}_4[x]}{\left\langle x^2 + x + 1 \right\rangle}$,
			$\frac{\mathbb{Z}_2[x, y]}{\left\langle x^2, xy, y^2 \right\rangle}$ & $K_3$ \\
			\hline
			\rule{0pt}{15pt} 4 & $\mathbb{Z}_{25}$, $\frac{\mathbb{Z}_5[x]}{\left\langle x^2 \right\rangle}$ & $K_4$ \\
			\hline
			\rule{0pt}{15pt} 6 &  $\mathbb{Z}_{49}$, $\frac{\mathbb{Z}_7[x]}{\left\langle x^2 \right\rangle}$  & $K_6$\\
			\hline
			
		\end{tabular}
		\vskip0.2cm
		\caption{}\label{table of local rings}
	\end{table}

	We now provide definitions of some topological aspects of graphs. A graph is said to be \emph{embeddable} on a two-dimensional connected topological space, known as surface, if it can be drawn without edge crossings, except at shared vertices. A graph \( G \) is \emph{planar} if such an embedding exists in the plane. An embedding is called \emph{outerplanar} when all vertices lie along the boundary of the unbounded region in the plane. The \emph{genus} of a graph \( G \) represents the minimum integer \( k \geq 0 \) such that \( G \) can be drawn on a surface having \( k \) handles (visualize a surface with \( k \) loops) and without any edge crossings, written as \( \gamma(G).\) In particular, graphs with genus \( 1 \) and genus \( 2 \) are termed \emph{toroidal} and \emph{double-toroidal}, respectively, while graphs with genus 0 are planar graphs. 
	For a non-negative integer \( k \), we can create a surface by starting with a sphere and attaching \( k \) crosscaps to it. Any compact non-orientable connected surface can be thought of as a surface similar to this, denoted as \( N_{k} \). The \textit{crosscap number} or non-orientable genus \( \overline{\gamma}(G) \) of a graph \( G \) is defined as the least \( k \) such that \( G \) can be drawn on \( N_k \) without edge crossings. If \( \overline{\gamma}(G) = 1 \), the graph is called a \textit{projective plane graph}; if \( \overline{\gamma}(G) = 2 \), it is termed a \textit{Klein-bottle graph}. Additionally, for any subgraph \( H \subseteq G \), it holds that \( \gamma(H) \leq \gamma(G) \) and \( \overline{\gamma}(H) \leq \overline{\gamma}(G) \). For a detailed overview, see \cite{twhite}.
	
	The subsequent results focus on the genus and crosscap of graphs, and their line graphs.
	
	\begin{lem}\label{lm1} \textnormal{\cite{chi-hsi}}
		Let \( G \) be a connected simple graph with vertex-disjoint connected subgraphs \( H_1 \) and \( H_2 \), where
		\begin{enumerate}[label=\textnormal{(\alph*).}]
			
			\item \( H_1 \) is planar and 3-connected,
			
			\item Each vertex of \( H_1 \) is adjacent to some vertex in \( H_2 \).
		\end{enumerate}
		Then \( \gamma(G) > \gamma(H_2) \) and \( \overline{\gamma}(G) > \overline{\gamma}(H_2) \).
	\end{lem}
	
	\begin{prop}
		
		\label{genus and crosscap of Kn and Knm} \textnormal{\cite{twhite}}
		For positive integers \(n\) and \(m\), we have:
		
		\vspace{0.2cm}
		\begin{enumerate}[label=\textnormal{(\roman*).}]
			
			\item If \(m \geq 3\):
			
			\begin{enumerate}[label=\textnormal{(\alph*).}]
				
				\item  $\gamma(K_m) = \lceil \frac{(m-4)(m-3)}{12} \rceil.$ 
				\vskip.2cm
				\item $\overline{\gamma}(K_m)=$ 
				$\begin{cases}
					\lceil \frac{(m-4)(m-3)}{6}\rceil & \text{if}\; m \neq 7\vspace{0.2cm}\\ 
					\qquad	3 & \text{if}\; m = 7.
				\end{cases}$
				
			\end{enumerate}
			\vspace{0.2cm}
			\item If \(n, m \geq 2\):
			
			\begin{enumerate}[label=\textnormal{(\alph*).}]
				\vskip.2cm
				\item  $\gamma(K_{n,m}) = \lceil \frac{(n-2)(m-2)}{4} \rceil$.
				\vskip.2cm
				\item  $\overline{\gamma}(K_{n,m})= \lceil \frac{(m-2)(n-2)}{2} \rceil$.
			\end{enumerate}  
		\end{enumerate} 
	\end{prop}
	\vspace{0.2cm}
	
	\begin{lem}\label{glkn}\textnormal{\cite{chi-hsi}}
		For positive integer $m$, we have:
		\vspace{0.2cm}
		\begin{enumerate}[label=\textnormal{(\roman*).}]  
			
			\item $\gamma(L(K_m)) \geq \frac{1}{12}(m-4)(m-3)(m+1),$  attains equality if  $m \equiv 0,3,4$ or  $7\, (mod \;12)$.
			\vskip0.2cm
			\item $\gamma(L(K_{2,m})) \geq \frac{1}{6}(m-3)(m-2) $, attains equality if $m \not\equiv 5$ or $ 9 \,(mod \;12)$.
			\vskip0.2cm
			\item $\gamma(L(K_{1,1,m})) = \gamma(L(K_{2,m+1}))$
		\end{enumerate} 	 
		
	\end{lem}
	\vskip0.2cm
	
	\begin{lem}\label{clkn} \textnormal{\cite{chi-hsi}}
		For positive integer $m$, we have:
		\vskip0.2cm
		\begin{enumerate}[label=\textnormal{(\roman*).}] 
			
			\item $\overline{\gamma}(L(K_m)) \geq \frac{1}{6}(m-4)(m-3)(m+1),$  attains equality if $ m \equiv 0,1,3$ or $4 \,(mod \;6)$ and $m \neq 7.$
			\vskip0.2cm
			
			\item $\overline{\gamma}(L(K_{2,m})) \geq \lceil \frac{1}{3}(m-3)(m-2) \rceil$, attains equality  if $m \neq 6$ and $m \not\equiv 1$ or $ 4 \,(mod \;6).$
			\vskip0.2cm
			
			\item $\overline{\gamma}(L(K_{1,1,m})) = \overline{\gamma}(L(K_{2,m+1}))$
		\end{enumerate} 
	\end{lem}
	
	The following results are used extensively throughout the paper.
	
	\begin{lem}\label{lm2}  \textnormal{\cite{chi-hsi}}
		For a simple graph $G$ with distinct vertices $v$ and $u$ of degree n and m respectively, the following holds:
		\begin{enumerate}[label=\textnormal{(\roman*).}] 
			
			\item \( \gamma(L(G)) \geq \gamma(K_n) + \gamma(K_m) \).
			
			\item \( \overline{\gamma}(L(G)) \geq \overline{\gamma}(K_n) + \overline{\gamma}(K_m) + \alpha \),  where $\alpha = 	 	\begin{cases} \; 0 & \text{if}\; (n-7)(m-7) \neq 0\\	  -1 & \text{otherwise}. \end{cases}$	
			
		\end{enumerate}
		
	\end{lem}
	
	\begin{lem}\label{lm3} \textnormal{\cite{chi-hsi}} 
		For a simple graph $G$ with three distinct vertices $u_1,u_2, u_3$ such that $\deg(u_1) \geq 5$, $\deg(u_2) \geq 5$  and $\deg(u_3) \geq 7,$ we have $\gamma(L(G)) \geq 3$\, and \, $\overline{\gamma}(L(G)) \geq 3$.
		
	\end{lem}
	
	\begin{rem}\label{remark for genus of some  bipartite graphs}\cite{mshariq, chi-hsi}\, The genus and the crosscap number of line graphs of some smaller complete bipartite graphs are as follows; $\gamma(L(K_{2,4})) = 1$,\,
		$\overline{\gamma}(L(K_{2,4})) =2$,
		$\gamma(L(K_{3,4})) = 2$,
		$\overline{\gamma}(L(K_{3,4})) \geq 3$,\, 
		$ \gamma(L(K_{2,5})) =2$,\,
		$\overline{\gamma}(L(K_{2,5})) =2$,\,
		$\gamma(L(K_{3,5})) \geq 3$,\, and \,
		$\overline{\gamma}(L(K_{3,5})) \geq 3$.
		\vskip.2cm\noindent
		We will use the values provided in the above remark in the following sections. 
	\end{rem}

	\section{Planarity and Outerplanarity of the line graph of essential graph}
	In this section, we focus on the characterization of all finite commutative rings $A$ whose essential graphs have planar line graphs.  We obtain a similar characterization for outerplanarity. For local rings, the result follow easily as in this case the essential graph of $A$ is a complete graph. For non-local case, the problem reduces to the case of zero-divisor graphs.

	We begin by studying planarity. The following results characterize the graphs whose line graphs are planar. 
	
	\begin{thm}\label{pl} \textnormal{\cite{green}}            
		A graph \( G \) has a planar line graph \( L(G) \) if and only if \( G \) contains no subdivision of any of the following graphs: \( K_{3,3}, \) \( K_{1,5}, \) \( P_4 \vee K_1,\) \( K_2 \vee \overline{K}_3.\)
		
	\end{thm}
	
	
	Recall that for a given ring $A,$  $EG(A)$  denotes the essential graph of $A.$ So, $L(EG(A))$ denotes the line graph of $EG(A).$ 
	\begin{thm}\label{pegl}
		Let $A$ be a finite local commutative ring, then the line graph of the essential graph of $A$ is planar 
		if and only if $A$ is isomorphic to one of the following 13 rings: 
		\(\mathbb Z_4, \hspace{.2cm}
		\mathbb Z_8,\hspace{.2cm}
		\mathbb Z_9,\hspace{.2cm}
		\mathbb Z_{25},\hspace{.2cm}
		\frac{\mathbb Z_2[x]}{\langle x^2 \rangle},\hspace{.2cm}
		\frac{Z_3[x]}{\langle x^2 \rangle},\hspace{.2cm} 		 
		\frac{\mathbb Z_2[x]}{\langle x^3 \rangle},\hspace{.2cm}		
		\frac{Z_5[x]}{\langle x^2 \rangle}, \hspace{.2cm}
		\frac{\mathbb Z_4[x]}{\langle x^3, x^2- 2 \rangle },\hspace{.2cm} 
		\frac{\mathbb F_3[x]}{ \langle x^2 \rangle},\hspace{.2cm}
		\frac{\mathbb Z_4[x]}{\langle 2x, x^2 \rangle}, \hspace{.2cm} \frac{\mathbb Z_2[x,y]}{\langle x^2, xy , y^2 \rangle},\hspace{.2cm} \frac{\mathbb Z_4[x]}{\langle x^2 + x + 1 \rangle}.\)
	\end{thm}
	
	\begin{proof}
		
		Since \( A \) is a finite local ring, its essential graph \( EG(A) \) forms a complete graph on \( n = |Z(A)^*| \) vertices, by Corollary~\ref{fcg}. When \( n \geq 5 \), the line graph \( L(EG(A)) \) contains \( P_4 \vee K_1 \) as a subgraph, which, by Theorem~\ref{pl}, guarantees that \( L(EG(A)) \) is non-planar. For \(n\leq 4 \), examining Table~\ref{table of local rings} shows that \( A \) must correspond to one of the thirteen rings listed in the theorem.
		
	\end{proof}
	
	We now consider the problem of the characterization of non-local rings whose essential graphs have planar line graphs.  Such characterization of rings is established by Chiang-Hsieh \cite{chi-hsi} for the planar line graphs of zero-divisor graphs.  
	
	\begin{lem}\label{pzg} \textnormal{\cite{chi-hsi}}
		Given a  finite non-local commutative ring $A,$ the line graph $L(\Gamma(A))$ of the zero-divisor graph of $A$ is planar if and only if either of the following holds: 
		\begin{enumerate}[label=\textnormal{(\roman*).}]
			\item $A$ is a reduced ring and isomorphic to one of the following $7$ rings: \( A \cong \mathbb{Z}_2 \times \mathbb{Z}_2 \),  
			\( \mathbb{Z}_2 \times \mathbb{Z}_3 \),  
			\( \mathbb{Z}_3 \times \mathbb{Z}_3 \),
			\( \mathbb{Z}_2 \times \mathbb{F}_4 \),  
			\( \mathbb{Z}_3 \times \mathbb{F}_4 \),
			\( \mathbb{Z}_2 \times \mathbb{Z}_5 \),   
			\( \mathbb Z_2 \times \mathbb Z_2 \times \mathbb Z_2.\)
			
			\item $A$ is a non-reduced ring and isomorphic to one of the following $4$ rings: $\mathbb Z_2\times \mathbb Z_4, \hspace{.2cm}
			\mathbb Z_3\times \mathbb Z_4,\hspace{.2cm}
			\mathbb Z_2 \times  \frac{\mathbb Z_2[x]}{\langle x^2 \rangle},\hspace{.2cm}  
			\mathbb Z_3\times \frac{\mathbb Z_2[x]}{\langle x^2 \rangle}.$
			
		\end{enumerate}
		
	\end{lem}

	\begin{thm}\label{pegnl}
		Let $A$ be a finite non-local commutative ring. Then, the line graph of the essential graph of $A$ is planar if and only if $A$ is isomorphic to one of the rings: \(\mathbb{Z}_2 \times \mathbb{Z}_2 \),  
		\( \mathbb{Z}_2 \times \mathbb{Z}_3 \),  
		\( \mathbb{Z}_3 \times \mathbb{Z}_3 \),
		\( \mathbb{Z}_2 \times \mathbb{F}_4 \),  
		\( \mathbb{Z}_3 \times \mathbb{F}_4 \),
		\( \mathbb{Z}_2 \times \mathbb{Z}_5 \),   
		\( \mathbb Z_2 \times \mathbb Z_2 \times \mathbb Z_2.\)
	\end{thm}
	
	\begin{proof}
		Since \( A \) is a finite ring that is not local, it can be expressed as a direct product \( A \cong A_1 \times A_2 \times \dots \times A_n \), where each \( A_i \) is a finite local ring. Note that $ n \geq 2.$ Let \( \mathfrak{m}_i \) denote the unique maximal ideal of \( A_i \), for each \( i= 1,2, \dots,n \). Let us first assume that the graph \( L(EG(A)) \) is planar. Suppose \(n\geq 2 \) and at least one of the components \( A_i \) is not a field. We may assume that \( A_1 \) is not a field, which ensures that \( \mathfrak{m}_1 \neq \{ 0 \} \). Then, there exists a nonzero element \( a \in \mathfrak{m}_1 \) such that \( \text{ann}_{A_1}(a) = \mathfrak{m}_1 \), which further implies \( a^2 = 0 \).
		Consider a set of vertices \( S = \{ s_1, s_2, s_3, s_4, r_1\} \), where  
		\( s_1 =(1,0, \dots, 0) \),  
		\( s_2 = (u,0,\dots,0) \), 
		\( s_3 = (0,1,0, \dots,0) \),  		  
		\( s_4 = (a,1,0,\dots,0) \),  
		\( r_1= (a,0,\dots,0) \),  
		with \( u \in A_1^\times \setminus \{1\}\). It follows that \( s_3 s_i = 0 = r_1 s_j \) for \( i \in \{1,2\} \) and \( j \in \{ 3,4\} \). Since \( \text{ann}_A(0) = A \), which is always essential, the annihilators \( \text{ann}_A(s_3 s_i) \) and \( \text{ann}_A(r_1 s_j) \) are essential ideals in \( A \) for the respective values of \( i \) and \( j \). Moreover, for \( k = 1,2 \), we observe that \( \text{ann}_A(s_4 s_k) = \text{ann}_A(r_1 s_k) = \mathfrak{m}_1 \times A_2 \times \dots \times A_n,
		\) which is also an essential ideal of \( A \). Thus, the subgraph induced by the vertices \( s_1, s_3, s_2, s_4 \) contains a path isomorphic to \( P_4 \). Since \( r_1 \) is adjacent to each of these vertices, the induced subgraph on the set \( S \) contains a subgraph isomorphic to \( P_4 \vee K_1 \), thereby contradicting the assumption of planarity by Theorem~\ref{pl}. Therefore, each \( A_i \) must be a field, and hence \( A \) is reduced. By Lemma~\ref{Zero divisor graph = essential graph}, it follows that \( EG(A) = \Gamma(A) \). Then, by applying Lemma~\ref{pzg}, we conclude that \( A \) is isomorphic to one of the rings listed in the statement of the theorem.
		\\
		\noindent Conversely, if \( A \cong \mathbb{Z}_2 \times \mathbb{Z}_2 \),  
		\( \mathbb{Z}_2 \times \mathbb{Z}_3 \),  
		\( \mathbb{Z}_3 \times \mathbb{Z}_3 \),
		\( \mathbb{Z}_2 \times \mathbb{F}_4 \),  
		\( \mathbb{Z}_3 \times \mathbb{F}_4 \),
		\( \mathbb{Z}_2 \times \mathbb{Z}_5 \),   
		\( \mathbb Z_2 \times \mathbb Z_2 \times \mathbb Z_2\), then \( A \) is reduced, and by Lemma~\ref{Zero divisor graph = essential graph}, it follows that \( EG(A) = \Gamma(A) \). Hence, by Lemma~\ref{pzg}(i), the line graph \( L(EG(A)) \) is planar.

		
	\end{proof}
	
	\begin{cor}\label{planatiry of essential grh for nonreduced}
		If \( A \) is a non-reduced ring, then \( L(EG(A)) \) is never planar.
		
	\end{cor}


	\vskip.2cm
	Now, we consider the problem of characterizing rings for which the line graphs of the essential graphs is outerplanar. We need the following result. 
	
	\begin{thm}\label{opl} \textnormal{\cite{hlin}}
		A graph \( G \) has an  outerplanar line graph \( L(G) \) if and only if \( G \) contains no subdivision of any of the following graphs:\, $K_{2,3}$,\, $K_{1,4}$,\, $P_3 \vee K_1 $.
	\end{thm}
	

	For the case of local rings, the result follows easily.
	
	\begin{thm}\label{oegl}
		Let $A$ be a finite local commutative ring. Then, the line graph of the essentail graph of $A$  is outerplanar if and only if $A$ is isomorphic to one of the following $11$ rings:
		\(\mathbb Z_4,\hspace{.2cm}
		\mathbb Z_9,\hspace{.2cm}
		\mathbb Z_8,\hspace{.2cm}
		\frac{\mathbb Z_2[x]}{\langle x^2 \rangle},\hspace{.2cm}
		\frac{Z_3[x]}{\langle x^2 \rangle},\hspace{.2cm}
		\frac{\mathbb Z_2[x]}{\langle x^3 \rangle},\hspace{.2cm} 
		\frac{\mathbb Z_4[x]}{\langle x^3, \, x^2- 2 \rangle },\hspace{.2cm} 
		\frac{\mathbb F_4[x]}{ \langle x^2 \rangle},\hspace{.2cm}
		\frac{\mathbb Z_4[x]}{\langle 2x, x^2 \rangle},\hspace{.2cm} 
		\frac{\mathbb Z_2[x,y]}{\langle x^2, xy , y^2 \rangle},\hspace{.2cm}
		\frac{\mathbb Z_4[x]}{\langle x^2 + x + 1 \rangle}.\)
	\end{thm}
	
	\begin{proof}
		Since \( A \) is a finite local commutative ring, Corollary \ref{fcg} implies that \( EG(A) \) is the complete graph \( K_n \) with \( n = |Z(A)^*|.\)  If  \( n \geq 4,\) then $K_n$ contains \( P_3 \vee K_1\) as a subgraph, and by Theorem \ref{opl}, this implies that \(L(EG(A))\) is not outerplanar.  For \( n \leq 3 \), the result follows from Table \ref{table of local rings}.  
	\end{proof}
	
	
	In the existing literature, the outerplanarity of the line graph of the zero-divisor graph has not been studied. Therefore, we begin by examining the outerplanarity of the line graph of the zero-divisor graph for finite non-local rings.

	\begin{thm}\label{outerplanarity of zero-divisor graph}
		Let $A$ be a finite non-local commutative ring. Then the line graph $L(\Gamma(A))$ of a zero-divisor graph of $A$  is outerplanar if and only if $A$ is isomorphic to one of the rings: 
		$\mathbb Z_2\times \mathbb Z_2,\hspace{.2cm}
		\mathbb Z_2\times \mathbb Z_3,\hspace{.2cm}
		\mathbb Z_3\times \mathbb Z_3,\hspace{.2cm}
		\mathbb Z_2\times \mathbb F_4,\hspace{.2cm} 
		\mathbb Z_2 \times \mathbb Z_4, \, \hspace{.2cm}
		\mathbb Z_2 \times \frac{ \mathbb Z_2[x]}{\langle x^2 \rangle},\hspace{.2cm}
		\mathbb Z_2 \times \mathbb Z_2 \times \mathbb Z_2.$
	\end{thm}
	
	\begin{proof} 
		Since \( A \) is a non-local finite ring, we can write \( A \cong A_1 \times A_2 \times \cdots \times A_n \), where each \( A_i \) is a finite local ring and \( n \geq 2 \). Let us assume that the line graph \( L(\Gamma(A))\) is outerplanar. If \( n \geq 4 \), consider the vertices   
		\( s_1 = (1,0,0,0,\ldots,0) \),  
		\( s_2 = (0,1,0,0,\ldots,0) \),  
		\( s_3 = (0,0,1,0,\ldots,0) \),  
		\( s_4 = (0,0,0,1,\ldots,0) \)  of $\Gamma(A).$
		Then, these vertices induce $K_4$ as a subgraph of \( \Gamma(A) \), which contains \( P_3 \vee K_1 \) as a subgraph, contradicting Theorem~\ref{opl}. Hence, \( n = 2 \) or \( n = 3 \).

		Suppose $n = 3$. We claim that \( |A_i| = 2 \) for each \( i = 1, 2, 3 .\) Assume that $|A_i| \geq 3$ for some $i.$  We may assume that $|A_1| \geq 3.$ Consider the vertices $s_1 = (1,0,0,\dots, 0),\,s_2 =(u,0,0,\dots,0),\, s_3 =(0,1,0,\dots,0),\,s_4 =(1,1,0,\dots,0),\,r_1 =(0,0,1,0,\dots,0),$ where \( u \in A_1^\times \setminus \{1\}.\) Then \( s_i r_1 = 0 \) for \( i \in \{ 1,2,3,4\} \), so the subgraph induced by these vertices contains \( K_{1,4} \) as a subgraph, contradicting Theorem~\ref{opl}. Thus, \( |A_i| = 2 \) for each \( i \), and hence \( A \cong \mathbb{Z}_2 \times \mathbb{Z}_2 \times \mathbb{Z}_2 \).
		
		Suppose \( n = 2 \), then we have \( A \cong A_1 \times A_2 \). Let us first assume that both \( A_1 \) and \( A_2 \) are fields. Then \( A \) is a reduced ring, therefore by Lemma~\ref{Zero divisor graph = essential graph}, \( \Gamma(A) = EG(A) \). According to Lemma~\ref{complete bipartite}, the graph \( \Gamma(A) \) is isomorphic to \( K_{|A_1|-1, |A_2|-1} \). Furthermore, Theorem~\ref{opl} ensures that \( K_{1,4} \) and \( K_{2,3} \) cannot appear as subgraphs in an outerplanar graph. This shows that \(|A_1| = 2\) and \(|A_2| \leq 4,\) or \(|A_{1}| = 3\) and \(|A_{2}| \leq 3\). Hence, the possible rings in this case are: 
		\(
		\mathbb Z_2\times \mathbb Z_2,\,
		\mathbb Z_2\times \mathbb Z_3,\,
		\mathbb Z_3\times \mathbb Z_3,\,
		\mathbb Z_2 \times \mathbb F_4.
		\)
		
		Now, suppose one of \( A_1 ,\)  \( A_2 \) is not a field. Assume, for instance, that \( A_2 \) is not a field. Let \( \mathfrak{m}_2 \) be the unique maximal ideal of \( A_2 \).  Suppose \( |\mathfrak{m}_2^*| \geq 2 \), so \( |A_2^\times| \geq 4 \). Consider the vertices $r_1 = (1,0),$ $s_1 = (0,u_1)$, $s_2 = (0,u_2)$, $s_3 = (0,u_3)$, $s_4 = (0,u_4)$,  where $u_1, u_2, u_3, u_4 \in A_2^\times.$ Observe that \( r_1 s_i = 0 \) for each \( i \in\{ 1, 2, 3, 4 \} \). Hence, the vertices \( \{r_1, s_1, s_2, s_3, s_4\} \) induce a subgraph in \( \Gamma(A) \) isomorphic to \( K_{1,4} \), which contradicts Theorem~\ref{opl}. Thus, \( |\mathfrak{m}_2^*| = 1 \), implying \( A_2 \cong \mathbb{Z}_4 \) or \( \frac{ \mathbb Z_2[x]}{\langle x^2 \rangle} \). Suppose \( |A_1| \geq 3 \). Choose \( u \in A_1^\times \setminus \{1\} \) and \( a \in \mathfrak{m}_2^* \) such that \( \text{ann}_{A_2}(a)= \mathfrak{m}_2\), hence \(a^2 = 0\). Consider the vertices 
		\(
		r_1 = (0,a),
		s_1 = (1,0), 
		s_2 = (u,0),            
		s_3 = (1,a), 
		s_4 = (u,a).            
		\)
		Then \( r_1s_i  = 0 \) for all \( i \in \{ 1,2,3,4 \}\). Therefore, the vertices \( \{r_1, s_1, s_2, s_3, s_4\} \) induce a subgraph in \( \Gamma(A) \) isomorphic to \( K_{1,4} \), again contradicting Theorem~\ref{opl}.  Hence, \( |A_1| = 2 \), so \( A_1 \cong \mathbb{Z}_2 \), and in this case the only possibile rings are:
		\(
		\mathbb{Z}_2 \times \mathbb{Z}_4, \,
		\mathbb{Z}_2 \times \frac{\mathbb Z_2[x]}{\langle x^2 \rangle}.
		\)
		Thus, if \( L(\Gamma(A)) \) is outerplanar, then \( A \) is isomorphic to one of the following rings:
		\(
		\mathbb{Z}_2 \times \mathbb{Z}_2, \quad
		\mathbb{Z}_2 \times \mathbb{Z}_3, \quad
		\mathbb{Z}_3 \times \mathbb{Z}_3, \quad
		\mathbb{Z}_2 \times \mathbb{F}_4, \quad
		\mathbb{Z}_2 \times \mathbb{Z}_4, \quad
		\mathbb{Z}_2 \times \mathbb{Z}_2[x]/\langle x^2 \rangle, \quad
		\mathbb{Z}_2 \times \mathbb{Z}_2 \times \mathbb{Z}_2.
		\)

		Conversely, 
		we verify that for each of these rings, \( L(\Gamma(A)) \) is outerplanar.
		\begin{itemize}	
			\item If \( A \cong \mathbb{Z}_2 \times \mathbb{Z}_2\quad\text{then}\quad \Gamma(A) = K_2\quad  \rightarrow\quad  L(\Gamma(A)) = K_1 \) is outerplanar.
			\item If \( A \cong \mathbb{Z}_2 \times \mathbb{Z}_3\quad \text{then}\quad  \Gamma(A) = P_3\quad  \rightarrow\quad L(\Gamma(A)) = K_2 \)  is outerplanar.
			\item If \( A \cong \mathbb{Z}_3 \times \mathbb{Z}_3\quad \text{then}\quad  \Gamma(A) = C_4\quad  \rightarrow\quad L(\Gamma(A)) = C_4 \) is outerplanar. 
			\item If \( A \cong \mathbb{Z}_2 \times \mathbb{F}_4\quad \text{then}\quad  \Gamma(A) = K_{1,3}\hspace{.2cm}  \rightarrow\quad L(\Gamma(A)) = K_3 \) is outerplanar.
			\item If \( A \cong \mathbb{Z}_2 \times \mathbb{Z}_4\quad \text{then}\quad  \Gamma(A) = \text{Fig.\ref{zero divisor graph of Z2 X Z4}(a)}\hspace{.2cm}  \rightarrow\hspace{.2cm} L(\Gamma(A)) =\text{ Fig.\ref{zero divisor graph of Z2 X Z4}(b)}\) is outerplanar.
			\item If \( A \cong \mathbb{Z}_2 \times \frac{\mathbb{Z}_2[x]}{\langle x^2 \rangle}\quad \text{then}\quad  \Gamma(A) \cong \text{Fig.\ref{zero divisor graph of Z2 X Z4}(a)}\hspace{.2cm} \rightarrow\hspace{.2cm} L(\Gamma(A)) \cong \text{ Fig.\ref{zero divisor graph of Z2 X Z4}(b)}\) is outerplanar.
			\item If \( A \cong \mathbb{Z}_2 \times \mathbb{Z}_2 \times \mathbb{Z}_2 \hskip.1cm \text{then}\hskip.1cm \Gamma(A) = \text{Fig.\ref{fig of z2 x z2 x z2}(a)}\hspace{.2cm}  \rightarrow\hspace{.2cm} L(\Gamma(A)) =\text{Fig.\ref{fig of z2 x z2 x z2}(b)}\) is outerplanar. \end{itemize}

	\end{proof}

	\begin{figure}[h!]
		\begin{tikzpicture}[scale= 0.6, font= \small]
			
			\filldraw (-0.5,1.5) circle (3pt);
			\filldraw (1.5,0) circle (3pt);
			\filldraw (1.5,1.5) circle (3pt);
			\filldraw (3.5, 1.5) circle (3pt);
			\filldraw (1.5, 3) circle (3pt);
			\draw[line width=0.2mm, black] (1.5,0) -- (-0.5,1.5);
			\draw[line width=0.2mm, black] (1.5,0) -- (1.5,1.5);
			\draw[line width=0.2mm, black] (1.5,0) -- (3.5,1.5);
			\draw[line width=0.2mm, black] (1.5,1.5) -- (1.5,3);
			
			\node[] at (1.5, -0.5) {\tiny(1,0)};
			\node[] at (-0.5,1.85) {\tiny(0,1)};
			\node[] at (2, 1.8) {\tiny(0,2)};
			\node[] at (3.5, 1.85) {\tiny(0,3)};
			\node[] at (1.5, 3.5) {\tiny(1,2)};
			
			\node[] at (0.25,0.65) {\tiny$e_1$};
			\node[] at (1.25,0.85) {\tiny$e_2$};
			\node[] at (2.75,0.65) {\tiny$e_3$};
			\node[] at (1.25,2.5) {\tiny$e_4$};
			
			\node[] at (1.5, -1.5) {(a)};
		\end{tikzpicture}   
		\hspace{2cm}
		\begin{tikzpicture}[scale = 0.8]
			
			\filldraw (0,0) circle (3pt);
			\filldraw (3,0) circle (3pt);
			\filldraw (1.5,1.5) circle (3pt);
			\filldraw (3.5,1.5) circle (3pt);

			\draw[line width=0.2mm, black] (0,0) -- (3,0);
			\draw[line width=0.2mm, black] (0,0) -- (1.5,1.5);
			\draw[line width=0.2mm, black] (3,0) -- (1.5,1.5);
			\draw[line width=0.2mm, black] (3,0) -- (3.5,1.5);
			
			\node[] at (0,-0.45) {\footnotesize$e_1$};
			\node[] at (3,-0.45) {\footnotesize$e_2$};
			\node[] at (1.5, 1.8) {\footnotesize$e_3$};
			\node[] at (3.5, 1.8 ) {\footnotesize$e_4$};
			
			\node[] at (1.5, -1.5) {(b)};
			
		\end{tikzpicture}
		
		\caption{(a): $\Gamma(\mathbb Z_2 \times \mathbb Z_4)$  \hspace{0.5cm} (b): $ L(\Gamma(\mathbb Z_2 \times \mathbb Z_4))$}
		\label{zero divisor graph of Z2 X Z4}
	\end{figure}
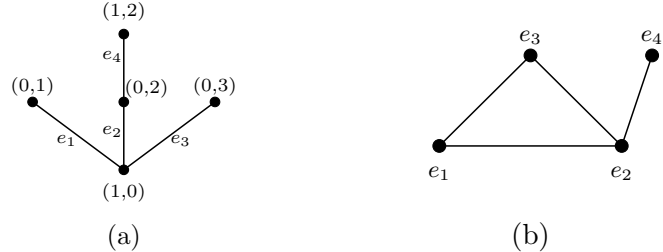

	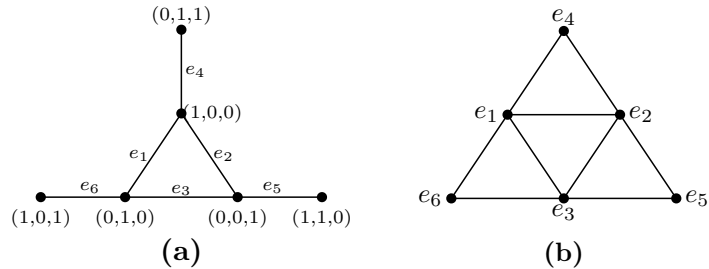
\begin{figure}[h!]
		\begin{center}
			\begin{tikzpicture}[scale=0.37, ultra thick]
				
				
				\filldraw (3,10) circle (3pt);
				\filldraw (3,7) circle (3pt);
				\filldraw (1,4) circle (3pt);
				\filldraw (5,4) circle (3pt);
				\filldraw (-2,4) circle (3pt);
				\filldraw (8,4) circle (3pt);
				
				\draw [line width=0.20mm, black] (3,7) -- (3,10);
				\draw [line width=0.20mm, black] (3,7) -- (1,4);
				\draw [line width=0.20mm, black] (3,7) -- (5,4);
				\draw [line width=0.20mm, black] (1,4) -- (-2,4);
				\draw [line width=0.20mm, black] (1,4) -- (5,4);
				\draw [line width=0.20mm, black] (5,4) -- (8,4);
				
				\node[font=\tiny] at (3,10.5) {(0,1,1)};
				\node[font=\tiny] at (4.1,7) {(1,0,0)};
				\node[font=\tiny] at (1,3.25) {(0,1,0)};
				\node[font=\tiny] at (5,3.25) {(0,0,1)};
				\node[font=\tiny] at (-2,3.25) {(1,0,1)};
				\node[font=\tiny] at (8,3.25) {(1,1,0)};
				
				\node[font=\tiny] at (1.5, 5.5) {$e_1$};
				\node[font=\tiny] at (4.5, 5.5) {$e_2$};
				\node[font=\tiny] at (3, 4.25) {$e_3$};
				\node[font=\tiny] at (3.5, 8.5) {$e_4$};
				\node[font=\tiny] at (6.25,4.35) {$e_5$};
				\node[font=\tiny] at (-0.25,4.35) {$e_6$};
				\node[] at (3,2) {\textbf{(a)}};

			\end{tikzpicture}
			\hspace{10pt}
			\begin{tikzpicture}[scale=0.37, ultra thick, font=\small]
				
				
				\filldraw (4,10) circle (3pt);
				\filldraw (2,7) circle (3pt);
				\filldraw (6,7) circle (3pt);
				\filldraw (0,4) circle (3pt);
				\filldraw (4,4) circle (3pt);
				\filldraw (8,4) circle (3pt);

				\draw [line width=0.20mm, black] (4,10) -- (2,7);
				\draw [line width=0.20mm, black] (4,10) -- (6,7);
				\draw [line width=0.20mm, black] (2,7) -- (6,7);
				\draw [line width=0.20mm, black] (2,7) -- (0,4);
				\draw [line width=0.20mm, black] (2,7) -- (4,4);
				\draw [line width=0.20mm, black] (6,7) -- (4,4);
				\draw [line width=0.20mm, black] (6,7) -- (8,4);
				\draw [line width=0.20mm, black] (0,4) -- (4,4);
				\draw [line width=0.20mm, black] (4,4) -- (8,4);

				\node[] at (1.25,7) {$e_1$};
				\node[] at (6.75,7) {$e_2$};
				\node[] at (4,3.5) {$e_3$};
				\node[] at (4,10.5) {$e_4$};
				\node[] at (8.75,4) {$e_5$};
				\node[] at (-0.75,4) {$e_6$};
				\node[] at (4, 2) {\textbf{(b)}};
				
			\end{tikzpicture}
			
			\caption{\textbf{(a) $\Gamma(\mathbb Z_2 \times \mathbb Z_2 \times \mathbb Z_2)$ \;  (b) $L(\Gamma(\mathbb Z_2 \times \mathbb Z_2 \times \mathbb Z_2))$}}\label{fig of z2 x z2 x z2}
			\label{Z2xZ2xZ2}
		\end{center}
		
	\end{figure}

	\newpage
	
	\begin{thm}\label{oegnl}
		Let $A$ be a finite non-local commutative ring. Then, the  line graph of the essential graph of $A$ is outerplanar if and only if $A$ is isomorphic to one of the rings:
		$\mathbb Z_2\times \mathbb Z_2,\hspace{.2cm}
		\mathbb Z_2\times \mathbb Z_3,\hspace{.2cm}
		\mathbb Z_3\times \mathbb Z_3,\hspace{.2cm} 
		\mathbb Z_2\times \mathbb F_4,\hspace{.2cm}
		\mathbb Z_2 \times \mathbb Z_2 \times \mathbb Z_2.$
	\end{thm}
	
	\begin{proof}
		We begin by assuming that the line graph \( L(EG(A)) \), associated with the essential graph of \( A \), is outerplanar.
		Since \( A \) is non-local, it can be written as \( A \cong A_1 \times A_2 \times \cdots \times A_n \), where each \( A_i \) is a finite local ring and \( n \geq 2 \). We aim to show that every \( A_i \) must be a field. 
		Assume, towards a contradiction, that atleast one of $A_i$ is not a field.  Assume that \( A_1 \) is not a field. Then we have an element \( a \in \mathfrak{m}_1^* \) such that its annihilator satisfies \( \operatorname{ann}_{A_1}(a) = \mathfrak{m}_1 \), where \( \mathfrak{m}_1 \) is the maximal ideal of \( A_1 \). Consider a set of vertices \( S = \{s_1, s_2, s_3, r\} \), where  
		\( s_1 = (1,0,0,\ldots,0) \),  
		\( s_2 = (a,0,0,\ldots,0) \),  
		\( s_3 = (u,0,0,\ldots,0) \),  
		\( r = (0,1,0,\ldots,0) \),  
		with \( u \in A_1^\times \setminus \{1\} \).
		Observe that \( s_i r = 0 \), implying $\text{ann}_A(s_ir)$ is an essential ideal for \( 1 \leq i \leq 3 \).
		Additionally, for \( j = 1,3 \), the product \( s_j s_2 \) has 
		\(
		\operatorname{ann}_A(s_j s_2) = \mathfrak{m}_1 \times A_2 \times \cdots \times A_n,
		\)
		which is also an essential ideal. Hence, the set of vertices \(\{s_1,s_2,s_3\}\) forms a path $P_3$ which together with $r$ induces $P_3 \vee K_1$ as a subgraph in $EG(A)$, contradicting Theorem~\ref{opl}.
		Thus, all \( A_i \) must be fields, and \( A \) is reduced. Lemma~\ref{Zero divisor graph = essential graph} then gives \( EG(A) = \Gamma(A) \). Applying Theorem~\ref{outerplanarity of zero-divisor graph}, we conclude that \( A \) is isomorphic to one of the rings:
		\(
		\mathbb{Z}_2 \times \mathbb{Z}_2, \, 
		\mathbb{Z}_2 \times \mathbb{Z}_3, \, 
		\mathbb{Z}_3 \times \mathbb{Z}_3, \,
		\mathbb{Z}_2 \times \mathbb{F}_4, \, 
		\mathbb{Z}_2 \times \mathbb{Z}_2 \times \mathbb{Z}_2.
		\)
		
		Conversely, for each of these rings, we have \( EG(A) = \Gamma(A) \), and \( L(\Gamma(A)) \) is outerplanar by Theorem~\ref{outerplanarity of zero-divisor graph}.
		
	\end{proof}

	
	\section{Genus of the line graph of essential graph}
	\noindent In this section, we provide a complete classification of all finite commutative rings with identity for which the genus of the line graphs of their essential graphs is atmost 2.
	
	We first consider the case of local rings. 
	

	
	
	\begin{thm}\label{genus of local ring}
		For a finite local commutative ring $A,$ the genus of the line graph of the essential graph of $A$  is never 1 or 2.
	\end{thm}
	\begin{proof}
		
		Let \( A \) be a finite, commutative local ring with maximal ideal \( \mathfrak{m} \). By Corollary~\ref{fcg} we have \( EG(A) \cong K_n \), where \( n = |\mathfrak{m}^*| \). According to Theorem~\ref{pegl}, the line graph \( L(EG(A)) \) is planar precisely when \( n \leq 4 \). 
		If \( n \geq 6 \), then Lemma~\ref{glkn} implies that the genus of \( L(EG(A)) \) satisfies \( \gamma(L(EG(A))) \geq 3 \). 
		Now consider the case \( n = 5 \). This would require \( |\mathfrak{m}| = 6 \). However, this is not possible as $A$ being a local ring, its order must be a power of a prime, and so must the size of its maximal ideal. Since 6 is not a power of any prime, such a local ring cannot exist. This completes the proof
		
	\end{proof}

	For the non-local rings case, we have the following result about the genus of the line graph associated to the zero-divisor graph. 
	
	\begin{lem}\label{gzg} \textnormal{\cite{chi-hsi}}
		Given a  finite non-local commutative ring $A,$ the following holds:
		\begin{enumerate}[label=\textnormal{(\roman*).}]
			
			\item The genus of $L(\Gamma(A))$ is 1 if and only if $A$ is a reduced ring and isomorphic to\,  
			$\mathbb F_4\times \mathbb F_4,$\,
			$\mathbb Z_3\times \mathbb Z_5$,\,
			$\mathbb Z_2\times \mathbb Z_7$,\,or\, 
			$\mathbb Z_2\times \mathbb F_8.$
			\vspace{0.1cm}
			\item The genus of $L(\Gamma(A))$ is 2 if and only if either of the following holds;
			\begin{enumerate}[label=\textnormal{(\alph*).}]
				\vskip.2cm
				\item $A$ is reduced ring and  isomorphic to one of the following rings: $\mathbb F_4 \times \mathbb Z_5,$\,
				$\mathbb Z_3\times \mathbb Z_7$,\,
				$\mathbb Z_2\times \mathbb F_9$,\, or 
				$\mathbb Z_2 \times \mathbb Z_2 \times \mathbb Z_3$.
				\vskip.1cm
				\item $A$ is non-reduced ring and  isomorphic to one of the following rings: 
				$\mathbb F_4 \times \mathbb Z_4,$\,\\
				$\mathbb Z_2\times \mathbb Z_8,$\,
				$\mathbb F_4 \times \frac{ \mathbb Z_2[x]}{\langle x^2 \rangle},$\,
				$\mathbb Z_2\times  \frac{\mathbb Z_2[x]}{\langle x^3 \rangle}$\,or\,  
				$\mathbb Z_2\times  \frac{\mathbb Z_4[x]}{\langle x^3, x^2 - 2 \rangle}.$ 
				
			\end{enumerate}
		\end{enumerate}
	\end{lem}  
	
	\newpage
	The following lemma is needed for the proof of the Lemma~\ref{glz34}.
	
	\begin{lem}\label{genus of K_1,2,m}
		For an integer \(m \geq 4,\) we have 
		
		\begin{enumerate}[label=\textnormal{(\roman*).}]
			\item \(\gamma(L(K_{1,2,m})) > \gamma(L(K_{2,m+1}))\)
			
			\item \(\overline{\gamma}(L(K_{1,2,m})) > \overline{\gamma}(L(K_{2,m+1}))\)
		\end{enumerate}   
	\end{lem}
	
	\begin{proof}
		We use Lemma \ref{lm1} to  establish the inequalities.  Let \(x\) be a vertex in \(K_{1,2,m}\) of degree \(m+1\). Since \(m \geq 4\), we can select four edges incident to \(x\), forming a \(K_4\) subgraph in the line graph \(L(K_{1,2,m})\). Observe that removing \(x\) from \(K_{1,2,m}\) result in the graph \(K_{1,1,m}\). Hence, the line graph \(L(K_{1,2,m})\) contains two disjoint connected subgraphs
		\(
		G_1 \cong K_4 \quad \text{and} \quad G_2 \cong L(K_{1,1,m}).
		\) 
		We now verify the hypotheses of Lemma~\ref{lm1}. The subgraph \(G_1\), isomorphic to \(K_4\), is both planar and 3-connected. Moreover, each edge incident to \(x\) in \(K_{1,2,m}\) has an endpoint that is also an endpoint of some edge in \(K_{1,1,m}.\) Therefore, by the definition of the line graph, each vertex of \(G_1\) is adjacent to some vertex in \(G_2 = L(K_{1,1,m}).\) Thus, the conditions of Lemma~\ref{lm1} are satisfied.		
		By applying Lemma~\ref{lm1} along with Lemma~\ref{glkn} and \ref{clkn}(iii), we obtain the following inequalities: $
		\gamma(L(K_{1,2,m})) > \gamma(L(K_{1,1,m})) =  \gamma(L(K_{2,m+1}))$ and $\overline{\gamma}(L(K_{1,2,m})) > \overline{\gamma}(L(K_{1,1,m})) = \overline{\gamma}(L(K_{2,m+1})).$ 	Thus, we conclude that 
		$\gamma(L(K_{1,2,m})) > \gamma(L(K_{2,m+1})) \quad \text{and} \quad \overline{\gamma}(L(K_{1,2,m})) > \overline{\gamma}(L(K_{2,m+1})).
		$
	\end{proof}

	We now prove two lemmas essential for the main results.
	
	\begin{lem}\label{glz34}
		If $A \cong \mathbb Z_{3} \times \mathbb Z_{4},$ then the genus and crosscap of the line graph $L(EG(A))$ are at least 3. 
		
	\end{lem}
	
	\begin{proof}
		The vertex set of the graph $EG(A)$ is the set  $\{(0, 1), (0, 2), (0, 3), (1, 0),(2, 0), (1, 2), (2, 2)\}$ of zero-divisors of \(A\) that are nonzero. Note that the product of $(0, 2)$ with $(0, 1)$ or $(0, 3)$ in $A$ gives $(0,2)$ and with the remaining four vertices it is $(0, 0).$ Since $\text{ann}_A(0, 0) = A,$ and   $\text{ann}_A(0, 2) = \mathbb{Z}_3 \times \mathfrak{m},$ where  $\mathfrak{m}$ is the maximal ideal of $\mathbb{Z}_4,$ it follows that $(0, 2)$ is adjacent to all remaining six vertices in $EG(A).$ Similarly, the four vertices $(1, 0),$ $(2, 0),$ $(1, 2),$  $(2, 2)$ are mutually non-adjacent but they are adjacent to the remaining three vertices. Moreover, the two vertices $(0,1)$ and $(0, 3)$ are mutually non-adjacent but are adjacent to the remaining five vertices. Thus, the graph $EG(A)$ is isomprphic to the complete tripartite graph $K_{1,2,4}.$ Then by the above Lemma \ref{genus of K_1,2,m} and Remark~\ref{remark for genus of some  bipartite graphs}, we have $\gamma(L(EG(A))) > \gamma(L(K_{2,5})) =2$ and $\overline{\gamma}(L(EG(A))) >  \overline{\gamma}(L(K_{2,5})) = 2.$ Thus, we obtain $\gamma(L(EG(A))) \geq 3$ and $\overline{\gamma}(L(EG(A))) \geq 3.$
	\end{proof}

	\vspace{5pt}

	\begin{lem}\label{glz24}
		If \( A \cong \mathbb Z_2 \times \mathbb Z_4\), then the genus and crosscap of the line graph of the essential graph of $A$ are both equal to 1.
	\end{lem}
	
	\begin{proof} 
		Let \(A \cong \mathbb Z_2 \times \mathbb Z_4\). Then the graph $EG(A)$ has five vertices $\{(0, 1), (0, 2), (0, 3), (1, 0), (1, 2)\},$ which represent the zero-divisors of \(A\) that are nonzero. 
		The structure of \(EG(A)\) and that of its line graph \(L(EG(A))\) are displayed in Figure~\ref{fig of Z2 x z4}(a) and (b), respectively.
		The subgraph of \(EG(A) \) obtained by deleting one edge between two vertices of degree three is isomorphic to \(  P_4 \vee K_1 \). Therefore, by Theorem \ref{pl}, it follows that \(L(EG(A))\) is not planar. However, it is clear from Figure \ref{fig of Z2 x z4}(b) that this line graph has edge crossing number one, implying that \( \gamma(L(EG(A))) = 1 \). Moreover, in Figure \ref{fig of Z2 x z4}(c), we give the projective embedding of \(L(EG(A)) \), which shows that  \( \overline{\gamma}(L(EG(A))) = 1\). 
	\end{proof}
	
	\begin{figure}[h!]
		\begin{center}
			
			\begin{tikzpicture}[scale=0.45, ultra thick, font=\small]
				
				
				
				\filldraw (2,18) circle (3pt);
				\filldraw (2,15) circle (3pt);
				\filldraw (2,12) circle (3pt);
				\filldraw (-2,15) circle (3pt);
				\filldraw (6,15) circle (3pt);		
				
				\draw[line width=0.20mm, black] (2,18) -- (-2,15);
				\draw[line width=0.20mm, black] (2,18) -- (2,15);
				\draw[line width=0.20mm, black] (2,18) -- (6,15);
				\draw[line width=0.20mm, black] (2,15) -- (-2,15);
				\draw[line width=0.20mm, black] (2,15) -- (6,15);
				\draw[line width=0.20mm, black] (2,18) -- (2,12);
				\draw[line width=0.20mm, black] (2,12) -- (-2,15);
				\draw[line width=0.20mm, black] (2,12) -- (6,15);

				
				\node[] at (2,18.5) {\tiny (1,0)};
				\node[] at (2.75,14.5) {\tiny(0,2)};
				\node[] at (2,11.5) {\tiny(1,2)};
				\node[] at (-3,15) {\tiny(0,1)};
				\node[] at (7,15) {\tiny (0,3)};

				
				\node[] at (0,17) {\tiny$e_1$};
				\node[] at (0.5,15.25) {\tiny$e_2$};
				\node[] at (0,13) {\tiny$e_3$};
				\node[] at (4,17) {\tiny$e_4$};
				\node[] at (3.75,15.25) {\tiny$e_5$};
				\node[] at (4,13) {\tiny$e_6$};
				\node[] at (2.35,16.25) {\tiny$e_7$};
				\node[] at (2.35,13.5) {\tiny$e_8$};
				\node[] at (2,10) {\textnormal{(a)}: $EG(\mathbb Z_2 \times \mathbb Z_4) $};

			\end{tikzpicture} 	
			\hspace{50pt}
			\begin{tikzpicture}[scale=0.43, ultra thick, font=\small]
				
				\filldraw (12,18) circle (3pt);	
				\filldraw (12,15) circle (3pt);
				\filldraw (12,12) circle (3pt);	
				\filldraw (16,13.5) circle (3pt);		
				\filldraw (16,16.5) circle (3pt);	
				\filldraw (20,18) circle (3pt);	
				\filldraw (20,15) circle (3pt);	
				\filldraw (20,12) circle (3pt);	
				
				\draw[line width= 0.20mm,black] (12,18) -- (20,18);
				\draw[line width= 0.20mm,black] (12,18) -- (12,15);
				\draw[line width= 0.20mm,black] (12,18) -- (16,16.5);
				\draw[line width= 0.20mm,black] (12,15) -- (20,15);
				\draw[line width= 0.20mm,black] (12,15) -- (12,12);
				\draw[line width= 0.20mm,black] (12,15) -- (16,13.5);
				\draw[line width= 0.20mm,black] (12,15) -- (16,16.5);
				\draw[line width= 0.20mm,black] (12,12) -- (16,13.5);
				\draw[line width= 0.20mm,black] (12,12) -- (20,12);
				\draw[line width= 0.20mm,black] (20,12) -- (16,13.5);
				\draw[line width= 0.20mm,black] (20,15) -- (16,13.5);
				\draw[line width= 0.20mm,black] (20,12) -- (20,15);
				\draw[line width= 0.20mm,black] (20,15) -- (16,16.5);
				\draw[line width= 0.20mm,black] (20,15) -- (20,18);
				\draw[line width= 0.20mm,black] (20,18) -- (16,16.5);
				\draw[line width= 0.20mm,black] (16,16.5) -- (16,13.5);
				\draw[line width= 0.20mm,black] (12,12) .. controls (10.75,13.5) and (10.75,16.5).. (12,18);
				\draw[line width= 0.20mm,black] (20,12) .. controls (21.25,13.5) and (21.25,16.5).. (20,18);

				
				\node[] at (12,18.5) {\tiny $e_1$};
				\node[] at (11.5,15) {\tiny$e_2$};
				\node[] at (12,11.5) {\tiny$e_3$};
				\node[] at (20,18.5) {\tiny$e_4$};
				\node[] at (20.5,15) {\tiny$e_5$};
				\node[] at (20,11.5) {\tiny$e_6$};
				\node[] at (16,17) {\tiny$e_7$};
				\node[] at (16,13) {\tiny$e_8$};
				\node[font= \small] at (16,10) {\textnormal{(b)}: $L(EG(\mathbb Z_2 \times \mathbb Z_4))$ };
				
			\end{tikzpicture} 
		\end{center}
	\end{figure}	
	\begin{figure}[h!]
		\begin{center}
			\begin{tikzpicture}[scale=0.4, ultra thick, font=\tiny]
				
				\draw [line width=0.20mm,black](10,1) circle (5);  
				
				\filldraw (10,6) circle (3pt);
				\filldraw (10,-4) circle (3pt);
				\filldraw (5,1) circle (3pt);
				\filldraw (15,1) circle (3pt);
				\filldraw (6.5,4.5) circle (3pt);
				\filldraw (6.5,-2.5) circle (3pt);
				\filldraw (13.5,4.5) circle (3pt);
				\filldraw (13.5,-2.5) circle (3pt);
				\filldraw (9.4,-0.75) circle (3pt);
				\filldraw (8.5,2.75) circle (3pt);
				\filldraw (11.2,3.8) circle (3pt);
				\filldraw (12.35,1.4) circle (3pt);
				\draw[line width= 0.20mm,black] (10,-4) -- (15,1);
				\draw[line width= 0.20mm,black] (6.5,-2.5) -- (6.5,4.5);
				\draw[line width= 0.20mm,black] (8.5,2.75) -- (6.5,4.5);
				\draw[line width= 0.20mm,black] (8.5,2.75) -- (10,6);
				\draw[line width= 0.20mm,black] (8.5,2.75) -- (11.2,3.8);
				\draw[line width= 0.20mm,black] (8.5,2.75) -- (9.4,-0.75);
				\draw[line width= 0.20mm,black] (11.2,3.8) -- (10,6);
				\draw[line width= 0.20mm,black] (11.2,3.8) -- (13.5,4.5);
				\draw[line width= 0.20mm,black] (11.2,3.8) -- (12.35,1.4);
				\draw[line width= 0.20mm,black] (12.35,1.4) -- (13.5,4.5);
				\draw[line width= 0.20mm,black] (12.35,1.4) -- (9.4,-0.75);
				\draw[line width= 0.20mm,black] (12.35,1.4) -- (15,1);
				\draw[line width= 0.20mm,black] (9.4,-0.75) -- (15,1);
				\draw[line width= 0.20mm,black] (9.4,-0.75) -- (6.5,4.5);

				
				\node[font=\small] at (10,6.5) {$e_5$}; 
				\node[font=\small] at (10,-4.5) {$e_5$}; 
				\node[font=\small] at (4.5,1) {$e_2$}; 
				\node[font=\small] at (15.5,1) {$e_2$}; 
				\node[font=\small] at (6,4.75) {$e_7$}; 
				\node[font=\small] at (6,-2.75) {$e_8$};     
				\node[font=\small] at (14,-2.75) {$e_7$}; 
				\node[font=\small] at (14,4.75) {$e_8$}; 
				\node[font=\small] at (9.4,-1.25) {$e_1$};
				\node[font=\small] at (9,2.5) {$e_4$};
				\node[font=\small] at (11,3.25) {$e_6$};
				\node[font=\small] at (11.8,1.6) {$e_3$};
				
				\node[font= \small] at (10,-6) {\textnormal{(c)}: Projective Embedding of $L(EG(\mathbb Z_2 \times \mathbb Z_4))$};

				
			\end{tikzpicture}
			
			\caption{\textbf{ (a), (b) and (c)}}\label{fig of Z2 x z4}
		\end{center}
	\end{figure}
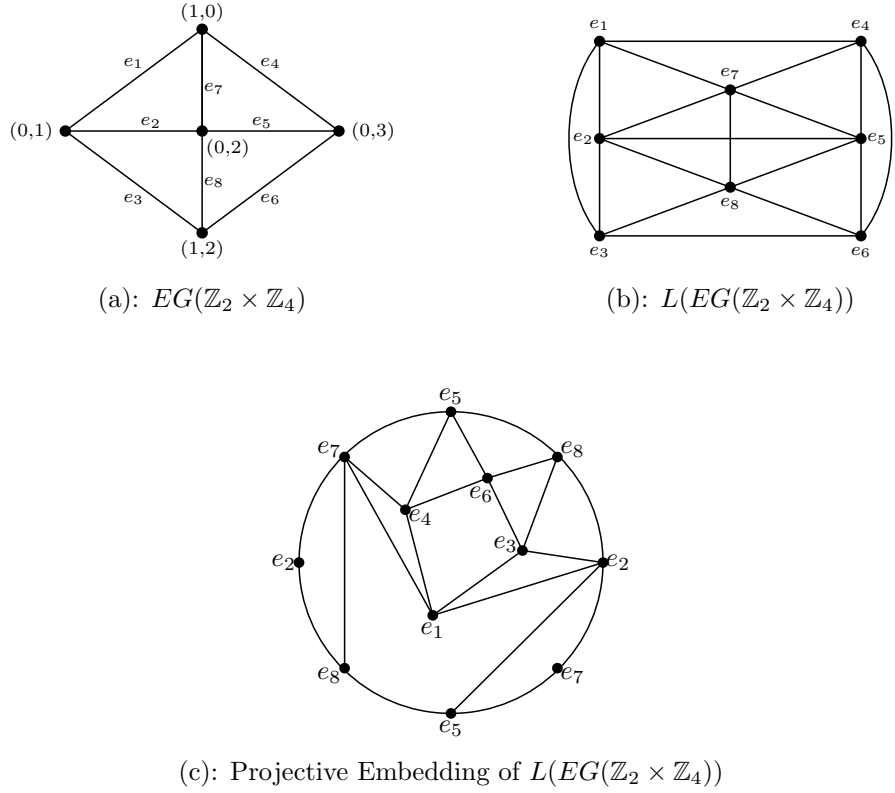
	


	
	\newpage
	We now present the main theorems of this section, which classify the rings whose line graphs \(L(EG(A))\) have genus at most two. We start with the case where the genus equals one.

	\begin{thm}\label{g1leg}
		Let $A$ be a finite non-local commutative ring. Then, the genus of the line graph of the essential graph $EG(A)$ is 1 if and only if $A$ is isomorphic to one of the six rings: 
		$\mathbb Z_2\times \mathbb Z_4,\hspace{.2cm}
		\mathbb Z_2 \times \frac{\mathbb Z_2[x]}{\langle x^2 \rangle},\hspace{.2cm}
		\mathbb F_4 \times \mathbb F_4, \hspace{.2cm}
		\mathbb Z_3 \times \mathbb Z_5,\hspace{.2cm}
		\mathbb Z_2 \times \mathbb Z_7,\hspace{.2cm}
		\mathbb Z_2 \times \mathbb F_8.$
	\end{thm}
	
	\begin{proof}
		Since \( A \) is a non-local finite ring, we can write \( A \cong A_1 \times A_2 \times \cdots \times A_n \), where each \( A_i \) is a finite local ring and \( n \geq 2 \). Let us assume that \(\gamma(L(EG(A))) = 1\). Suppose, for contradiction, that \(n \geq 4\). Consider the vertices 
		\( s_1 = (1,0,0,0,\ldots,0),\;
		s_2 = (0,1,0,0,\ldots,0),\;
		s_3 = (0,0,1,0,\ldots,0),\;
		s_4 = (0,0,0,1,0,\ldots,0),\; 
		s_5 = (0,1,1,0,\ldots,0),\;
		s_6 = (0,0,1,1,0,\ldots,0),\;
		s_7 = (1,0,1,0,\ldots,0).\)
		We observe that \(s_1\) is adjacent to \(s_i\) for all \(2 \leq i \leq 6\), and \(s_2\) is adjacent to \(s_j\) for \(j \in \{1, 3, 4, 6, 7\}\), since their pairwise products are zero  and \( \mathrm{ann}_A(0) = A \).
		Thus, both \(s_1\) and \(s_2\) have degree at least 5 in \(EG(A)\). Applying Lemma~\ref{lm2}, it follows that $\gamma(L(EG(A))) \geq \gamma(K_5) + \gamma(K_5) = 2$, which contradicts the assumption. Hence, we must have \(n \leq 3\). 
		
		Suppose \(n = 3\), and assume that \(|A_i| \geq 3\) for some index \(i\).  Assume, for instance, that
		$|A_1| \geq 3$. Consider the vertices
		\(s_1 = (1,0,0)\),\; 
		\(s_2 = (0,1,0)\),\;
		\(s_3 = (0,0,1)\),\;
		\(s_4 = (v,0,0)\),\;
		\(s_5 = (1,0,1)\),\; 
		\(s_6 = (v,0,1)\),\;
		\(s_7 = (1,1,0)\)\; and 
		\(s_8 = (v,1,0)\), where \(v \in A_1^\times \setminus \{1\}.\) Note that \(s_2s_i =0 \)  for \(i \in \{1, 3, 4, 5, 6\}\) and \(s_3s_j = 0\) for \(j \in \{ 1, 2, 4, 7, 8\}\), furthermore \( \mathrm{ann}_A(0) = A \) is essential. Thus, both \(\deg(s_2)\) and \(\deg(s_3)\) are at least 5. By Lemma~\ref{lm2}, it follows that \(
		\gamma(L(EG(A))) \geq \gamma(K_5) + \gamma(K_5) = 2,
		\)
		leading to a contradiction. Hence, \(|A_i| \leq 2\) for all \(i\), which implies \(A \cong \mathbb{Z}_2 \times \mathbb{Z}_2 \times \mathbb{Z}_2\). However, by Theorem~\ref{pegnl}, in this case \(L(EG(A))\) is planar, again leading to a contradiction.
		
		Now, suppose \(n = 2\). Let us begin with the assumption that both \(A_1\) and \(A_2\) are not fields. Then, for each \(i = 1,2\), we have \(|A_i| \geq 4\) and \(|\mathfrak{m}_i^*| \geq 1\). Moreover, for each \(i\), there exists \(a_i \in \mathfrak{m}_i^*\) such that \(\text{ann}_{A_i}(a_i) = \mathfrak{m}_i\), implying \(a_i^2 = 0\). Consider the set of vertices
		\(
		S = \{s_1, s_2, s_3, r_1, r_2, r_3, r_4\},
		\)
		where \(s_1 = (1,0)\), \(s_2 = (u_1,0)\), \(s_3 = (a_1,0)\), \(r_1 = (0,1)\), \(r_2 = (0,u_2)\), \(r_3 = (0,a_2)\), and \(r_4 = (a_1,1)\), with \( u_i \in A_i^\times \setminus \{1\}\). We observe that \(s_i r_k = 0\) for all \(i \neq k\) and \(1\leq i,k \leq 3\) and \(s_3 r_4 = 0\). Moreover, for \(l \in\{ 1,2\},\)  we have \(\text{ann}_A(s_l r_4) = \mathfrak{m}_1 \times A_2\), which is an essential ideal. Therefore, the subgraph of \(EG(A)\) induced by \(S\) contains a copy of \(K_{3,4}\). By Remark~\ref{remark for genus of some bipartite graphs}, it follows that \(\gamma(L(EG(A))) > 1\), leading to a contradiction. Thus, at least one of the \(A_i\) must be a field.

		Now, consider the case where \( A_1 \) is a field while \( A_2 \) is not a field. We first prove that $|\mathfrak{m}_2^* | =1.$  Suppose $|\mathfrak{m}_2^*| \geq 2$. Then we have $|A_2^{\times}| \geq 4$ and there exist $a, b \in \mathfrak{m}_2^*$  with \(ab = 0\) and \(\mathrm{ann}_{A_2}(a) = \mathfrak{m}_2\),  implying \(a^2 = 0\).
		Consider the vertices \(
		s_1 = (0,a),\;
		s_2 = (1,a),\;
		r_1 = (0,b),\;
		r_2 = (0,1),\;
		r_3 = (0,u_1),\;
		r_4 = (0,u_2),\;
		r_5 = (0,u_3),
		\) where \( u_1, u_2, u_3 \in A_1^\times \setminus \{1\} \). Observe that for each \( i \in \{1,2\} \), the annihilator \(s_i r_1 = 0 \), and for \( 2 \leq j \leq 5 \), we have \( \mathrm{ann}_A(s_i r_j) = A_1 \times \mathfrak{m}_2\), which is essential. This ensures that \( \deg(s_i) \geq 5 \) for both \( i\in \{1,2\}.\) Consequently, applying Lemma~\ref{lm2}, we obtain that
		\(
		\gamma(L(EG(A))) \geq \gamma(K_5) + \gamma(K_5) = 2,
		\)
		which contradicts the assumption. Thus, \( |\mathfrak{m}_2^*| = 1 \), implying \( A_2 \cong \mathbb{Z}_4 \) or $\frac{\mathbb{Z}_2[x]}{\langle x^2 \rangle}$. 
		Next, suppose \( |A_1| \geq 3 \). Let us define the set
		\(
		S' = \{s_1, s_2, s_3, r_1, r_2, r_3, r_4\},
		\)
		where
		\( s_1 = (0,1),\;
		s_2= (0,u_1),\;
		s_3 = (0,a),\;
		r_1 = (1,0),\;
		r_2 = (u_2,0),\;
		r_3 = (1,a),\;
		r_4 = (u_2,a),
		\)
		with \( u_i \in A_i^\times \setminus \{1\}.\) Observe that \(s_i r_k = 0\) for \(i \in \{1,2,3\}\) and \(k \in \{ 1,2\}\), and \(s_3 r_j = 0\) for \(j \in\{3,4\}\). Additionally, for \(l \in \{ 1,2\}\) and \(m \in \{  3,4\}\), we have \(\text{ann}_A(s_l r_m) = A_1 \times \mathfrak{m}_2\), which is essential. Therefore, the induced subgraph of \( EG(A) \) on \( S' \) contains \( K_{3,4} \) as subgraph, and therefore, by Remark~\ref{remark for genus of some bipartite graphs}, we have 
		\(
		\gamma(L(EG(A))) \geq \gamma(L(K_{3,4})) = 2,
		\)
		again contradicting the assumption. Hence, we conclude that \( |A_1| = 2 \), so \( A_1 \cong \mathbb{Z}_2 \). The proof holds similarly if $A_2$ is a field and $A_1$ is not a field. Thus, the only possible rings in this case are $A \cong \mathbb{Z}_2 \times \mathbb{Z}_4$ or $\mathbb{Z}_2 \times \frac{\mathbb{Z}_2[x]}{\langle x^2 \rangle}.$
		
		Suppose that both \(A_1\) and \(A_2\) are fields. In this case, by Lemma~\ref{Zero divisor graph = essential graph}, we have \(EG(A) = \Gamma(A)\).
		Consequently, it follows from Lemma~\ref{gzg}(i) that
		\(
		A \cong \mathbb{F}_4 \times \mathbb{F}_4,\;
		\mathbb{Z}_3 \times \mathbb{Z}_5,\;
		\mathbb{Z}_2 \times \mathbb{Z}_7,\; \text{or}\;
		\mathbb{Z}_2 \times \mathbb{F}_8.
		\)
		
		Conversely, if \( A \cong \mathbb{Z}_2 \times \mathbb{Z}_4 \) or \( \mathbb{Z}_2 \times \frac{\mathbb{Z}_2[x]}{\langle x^2 \rangle} \). We can observe that in both cases, the essential graphs coincide. Therefore, by Lemma~\ref{glz24}, we obtain \( \gamma(L(EG(A))) = 1 \). 
		Now, consider the case where \( A \cong \mathbb{F}_4 \times \mathbb{F}_4 \), \( \mathbb{Z}_3 \times \mathbb{Z}_5 \), \( \mathbb{Z}_2 \times \mathbb{Z}_7 \), or \( \mathbb{Z}_2 \times \mathbb{F}_8 \). For each of these rings, Lemma~\ref{Zero divisor graph = essential graph} shows that the essential graph \( EG(A) \) is equal to the zero-divisor graph \( \Gamma(A) \). Then, by Lemma~\ref{gzg}(i), we have \( \gamma(L(EG(A))) = 1 \).
	\end{proof}
	\vspace{2pt}
	
	Next, we prove that when the genus of the line graph $L(EG(A)$ is two, then $A$ is a reduced ring and $EG(A) 	= \Gamma(A).$ Thus, from Lemma \ref{gzg}(ii), we get  the  characterization of rings $A$ for which the genus of the graph $L(EG(A)$ is 2.

	\begin{thm}\label{g2leg}
		Let $A$ be a finite non-local commutative ring. Then, the genus of the line graph of the essential graph $EG(A)$ is 2 if and only if $A$ is isomorphic to one of the rings: $\mathbb F_4\times \mathbb Z_5,\hspace{.2cm}
		\mathbb Z_3\times \mathbb Z_7,\hspace{.2cm}
		\mathbb Z_2\times \mathbb Z_9,\hspace{.2cm}
		\mathbb Z_2 \times \mathbb Z_2 \times \mathbb Z_3.$
	\end{thm}
	\begin{proof}
		Since \( A \) is a finite non-local ring, it  decomposes into a product \( A \cong A_1 \times A_2 \times \cdots \times A_n \), where each \( A_i \) is a finite local ring with maximal ideal \( \mathfrak{m}_i \) and $n\geq 2$. Assume that \(\gamma(L(EG(A))) = 2.\)
		We aim to show that \( n = 2 \) or \( 3 \). Suppose, to the contrary, that \( n \geq 4.\) Consider the vertices 
		$s_1 = (1,0,0,0,\ldots,0)$, 
		$s_2 = (0,1,0,0,\ldots,0)$, 
		$s_3 = (0,0,1,0,\ldots,0)$,
		$s_4 = (1,0,0,1,0,\ldots,0)$, 
		$s_5 = (0,1,1,0,\ldots,0)$, 
		$s_6 = (0,1,0,1,0,\ldots,0)$, 
		$s_7 = (0,0,1,1,0,\ldots,0)$, 
		$s_8 = (1,0,1,0,\ldots,0)$, 
		$s_9 = (1,0,0,1,0, \ldots, 0)$, 
		$s_{10} = (1,1,0,0,\ldots,0)$,
		$s_{11} = (0,1,1,1,0,\ldots,0)$,
		$s_{12} = (1,0,1,1,0,\ldots,0)$ and 
		$s_{13} = (1,1,0,1,0,\ldots,0).$ 
		Note that \( s_1s_i = 0 \) for \( i \in \{2,3,5,6,7,11\}, \) \( s_2s_j = 0 \) for \( j \in \{ 1,3,4,7,8,9,12\}, \) and \( s_3s_k = 0 \) for \( k \in \{ 1,2,4,6,9,10,13\}.\) Therefore, we find that \( \deg(s_1) \geq 6 \), \( \deg(s_2) \geq 7 \), and \( \deg(s_3) \geq 7 \). Applying Lemma~\ref{lm3}, it follows that \( \gamma(L(EG(A))) \geq 3 \), leading to a contradiction. Thus $ n = 2$ or $ 3.$ 
		
		Suppose \( n = 3 \). We aim to show that each component ring \( A_i \) must be a field. Assume, for instance, that \( A_1 \) is not a field. Then, there exists \( a \in \mathfrak{m}_1^*\) which satiesfies \( \mathrm{ann}_{A_1}(a) = \mathfrak{m}_1.\)
		Consider the vertices $s_1 = (1,0,0)$, $s_2 = (0,1,0)$, $s_3 = (0,0,1)$, $r_1 = (a,0,0)$, $r_2 = (a,1,0)$, $r_3 = (a,0,1)$, $r_4 = (a,1,1)$, $r_5 = (0,1,1)$, $r_6 = (1,0,1)$, $r_7 = (1,1,0)$. We observe that $s_is_j =0$  for distinct \( i, j \in \{1,2,3\} \),  
		\( s_1 r_5 = 0 \),  
		\( s_2 r_h = 0 \) for \( h \in \{1,3,6\} \) and  
		\( s_3 r_w = 0 \) for \( w \in \{1,2,7\} \). Furthermore, for each \( l \in \{1,2,3,4\} \), \(\text{ann}_A(s_1 r_l) = \mathfrak{m}_1 \times A_2 \times A_3,\) is an essential ideal. As a result, we have \( \deg(s_1) \geq 7 \), \( \deg(s_2) \geq 5 \), and \( \deg(s_3) \geq 5 \). Applying Lemma~\ref{lm3}, we obtain \( \gamma(L(EG(A))) \geq 3 \), leading to a contradiction. Since the choice of component ring is arbitrary, we have proved that each \( A_i \) must be a field. Now, by Lemma~\ref{Zero divisor graph = essential graph}, we have \( EG(A) = \Gamma(A) \). Hence, applying Lemma~\ref{gzg}, it follows that \( A \cong \mathbb{Z}_2 \times \mathbb{Z}_2 \times \mathbb{Z}_3.\)
		
		Now, suppose \( n = 2 \).  We aim to show that both \( A_1 \) and \( A_2 \) are fields. Let us first assume both are not fields. Consider the set $S = \{s_1, s_2, s_3, r_1, r_2, r_3, r_4, r_5\}$, where
		$s_1 = (1, 0)$, $s_2 = (u_1, 0)$, 
		$s_3 = (a_1, 0)$, $r_1 = (0, 1)$, 
		$r_2 = (0, u_2)$, $r_3 = (0, a_2)$, 
		$r_4 = (a_1, 1)$, $r_5 = (a_1, u_2)$, where
		$u_i \in A_i^\times$,  and $a_i \in \mathfrak{m}_i^*$ for which \( \mathrm{ann}_{A_i}(a) = \mathfrak{m}_i\) for each $i = 1,2$.  We observe that $s_i r_j = 0$ for $1 \leq i,j \leq 3$, $s_3 r_k = 0$ for \(k \in \{4,5\}\). Furthermore, $\text{ann}_A(s_i r_j) = \mathfrak{m}_1 \times A_2$ for \(i \in \{1,2\}\) and \(j \in \{4,5\},\) is essential. Therefore, the set \( S \) induces a subgraph in \( EG(A) \) that contains \( K_{3,5} \), implying that \( \gamma(L(EG(A))) \geq 3 \) by Remark \ref{remark for genus of some  bipartite graphs}, which leads to a contradiction. Thus, atleast one $A_i$ must be a field.

		Now, consider the case where \( A_1 \) is a field while \( A_2 \) is not a field. We first prove that $|\mathfrak{m}_2^* | =1.$  Suppose $|\mathfrak{m}_2^*| \geq 2$. Then we have $|A_2^{\times}| \geq 4$ and there exist $a, b \in \mathfrak{m}_2^*$  with \(ab = 0\) and \(\mathrm{ann}_{A_2}(a) = \mathfrak{m}_2\),  implying \(a^2 = 0\). Consider the vertex set $S' = \{s_1, s_2, s_3, r_1, r_2, r_3, r_4, r_5\}$, where 
		$s_1 = (1,0)$, $s_2 = (1,a)$,
		$s_3 = (0,a)$, $r_1 = (0,u_1)$, 
		$r_2 = (0,u_2)$, $r_3 = (0,u_3)$,
		$r_4 = (0,u_4)$, $r_5 = (0,b)$ and 
		$u_i \in A_2^\times$ for $1 \leq k \leq 4$.  We observe that $s_1 r_j = 0$ for $1 \leq j \leq 5$, $s_2 r_5 = 0 = s_3 r_5$. Furthermore, $\text{ann}_A(s_i r_j) = A_1 \times \mathfrak{m}_2$ for \(i \in \{2,3\}\) and $1 \leq j \leq 4$ is an essential ideal. Therefore, the set \( S' \) induces a subgraph in \( EG(A) \) that contains \( K_{3,5} \), implying that \( \gamma(L(EG(A))) \geq 3 \) by Remark \ref{remark for genus of some  bipartite graphs}, a contradiction. 
		Thus $|\mathfrak{m}_2^*| = 1$, which implies that $A_2 \cong \mathbb{Z}_4$ or \(\frac{\mathbb{Z}_2[x]}{\langle x^2 \rangle}\). Now suppose that $|A_1| \geq 4$. Consider the vertex set $S'' = \{ s_1, s_2, s_3, r_1, r_2, r_3, r_4, r_5 \}$, where 
		$s_1 = (0,1)$, $s_2 = (0,v_1)$, 
		$s_3 = (0,v_2)$, $r_1 = (u_1,0)$, 
		$r_2 = (u_2,0)$, $r_3 = (u_3,0)$,
		$r_4 = (u_1,a)$, $r_5 = (u_2,a)$, 
		with $a \in \mathfrak{m}_2^*$, $v_1, v_2 \in A_2^\times$ and $u_k \in A_1^\times$ for all $1 \leq k \leq 3$. We observe that $s_i r_j = 0$ for  all \(i,j \in \{1,2,3\}\), also $\text{ann}_A(s_l r_h) = A_1 \times \mathfrak{m}_2$ is an essential ideal for all $1 \leq l \leq 3$ and \(h \in \{4,5\}.\) Therefore, the set \( S'' \) induces a subgraph in \( EG(A) \) that contains \( K_{3,5} \), implying that \( \gamma(L(EG(A))) \geq 3 \) by Remark \ref{remark for genus of some  bipartite graphs}, leads to a contradiction. Hence, $|A_1| \leq 3$, that is, $A_1 \cong \mathbb{Z}_2$ or $\mathbb{Z}_3$. 
		If \( A \cong \mathbb{Z}_2 \times \mathbb{Z}_4 \) or \( \mathbb{Z}_2 \times \frac{\mathbb{Z}_2[x]}{\langle x^2 \rangle} \), then \( \gamma(L(EG(A))) = 1 \) by Lemma~\ref{glz24}, contradicts the assumption. Similarly, if \( A \cong \mathbb{Z}_3  \times \mathbb{Z}_4 \) or \(\mathbb{Z}_3  \times  \frac{\mathbb{Z}_2[x]}{\langle x^2 \rangle} \), then \( \gamma(L(EG(A))) \geq 3 \) by Lemma~\ref{glz34}, again a contradiction. Similarly, if $A_1$ is not a field but $A_2$ is field, we get a contradiction. Thus, both $A_1$ and $A_2$ are fields. Therefore, by  Lemma \ref{Zero divisor graph = essential graph}, $\Gamma(A) = EG(A)$ and hence,  $A \cong \mathbb F_4\times \mathbb Z_5,\hspace{.2cm}
		\mathbb Z_3\times \mathbb Z_7,$\, or\,
		$\mathbb Z_2\times \mathbb Z_9,$ by  Lemma \ref{gzg} (ii).
		
		Conversly, if $A\cong \mathbb Z_2\times \mathbb F_9$,\, $\mathbb Z_3\times \mathbb Z_7$,\, $\mathbb F_4 \times \mathbb Z_5$ or $\mathbb Z_2 \times \mathbb Z_2 \times \mathbb Z_3$, then $A$ is reduced. Hence, the result follows by Lemma \ref{Zero divisor graph = essential graph} and \ref{gzg}(ii)(a).
	\end{proof}
	\begin{cor}
		For a non-reduced ring $A$, the genus of $L(EG(A))$ is never 2.
	\end{cor}

	
	\section{Crosscap of the line graph of essential graph}
	\noindent In this section, we provide a complete classification of all finite commutative rings with identity for which the crosscap number of the line graphs of their essential graphs is at most 2.
	
	As in the previous section, we begin by considering the case of local rings.
	\begin{thm}
		For a finite local commutative ring $A,$ the crosscap number of the line graph $L(EG(A)) $ is never 1 or 2. 
	\end{thm}
	\begin{proof}
		The result can be obtained using the approach similar to the one used in the proof of Theorem \ref{genus of local ring}.
		
	\end{proof}
	
	For non-local rings, we make use of the following result on crosscap number of the line graph of the zero-divisor graphs.
	\begin{lem}\label{czg} \textnormal{\cite{chi-hsi}}
		Given a  finite non-local commutative ring $A,$ the following holds:
		\vskip.2cm
		\begin{enumerate}[label=\textnormal{(\roman*).}]
			
			\item The crosscap number of $L(\Gamma(A))$ is 1 if and only if $A$ is a reduced ring and  isomorphic to\, $\mathbb Z_2\times \mathbb Z_7$\, or\, $\mathbb F_4\times \mathbb F_4$.
			\vspace{0.1cm}
			\item The crosscap number of $L(\Gamma(A))$ is  2 if and only if either of the following holds:
			\vskip.2cm
			\begin{enumerate}[label=\textnormal{(\alph*).}]
				
				\item $A$ is a reduced ring and isomorphic to $\mathbb Z_3 \times \mathbb Z_5$ or\, $\mathbb Z_2 \times \mathbb Z_2 \times \mathbb Z_3$.
				
				\item $A$ is a non-reduced ring and isomorphic to \, $\mathbb F_4 \times \mathbb Z_4 $ or\, $\mathbb F_4 \times 
				\frac{ \mathbb Z_2[x]}{\langle x^2 \rangle}.$ 
			\end{enumerate}
			
		\end{enumerate}
	\end{lem}
	We start with the case where the crosscap number equals one.
	
	\begin{thm}\label{c1leg}
		Let $A$ be a finite non-local commutative ring. Then, the crosscap number of the line graph of the essential graph of $A$ is 1 if and only if $A$ is isomorphic to one of the rings: $\mathbb Z_2\times \mathbb Z_4,\hspace{.2cm} \mathbb Z_2 \times \frac{\mathbb Z_2[x]}{\langle x^2 \rangle},\hspace{.2cm} \mathbb Z_2\times \mathbb Z_7,\hspace{.2cm} \mathbb F_4\times \mathbb F_4.$
	\end{thm}
	\begin{proof}
		Since \(A \) is a non-local finite ring, we can write \( A \cong A_1 \times A_2 \times \cdots \times A_n \), where each \( A_i \) is a finite local ring and \( n \geq 2 \). Let us assume that \(\overline{\gamma}(L(EG(A))) = 1\).   We claim that $ n = 2.$ Suppose, for contradiction, that \(n \geq 3.\) Additionally, assume that for at least one \( i \), $|A_i| \geq 3.$  We may assume that $|A_1| \geq 3.$ Consider the vertices  
		$s_1 = (1,0,0,\dots,0),\, 
		s_2 = (0,1,0,\dots,0),\, 
		s_3 = (0,0,1,0,\dots,0),\,
		s_4 = (u,0,0,\dots,0),\, 
		s_5 = (u,1,0,\dots,0),\,
		s_6 = (u,0,1,0,\dots,0),\,
		s_7 = (1,0,1,0,\\\dots,0),\,
		s_8 = (1,1,0,\dots,0),$ 
		where $ u \in A_1^\times \setminus \{1\}$. We observe that $s_2s_j=0$ for $i\in \{1,3,4,6,7\}$ and $s_2s_j = 0$ for $j \in \{1,2,4,5,8\}$.
		Hence, \( \deg(s_2) \geq 5 \) and \( \deg(s_3) \geq 5 \). By Lemma~\ref{lm2}, it follows that \(
		\gamma(L(EG(A))) \geq \gamma(K_5) + \gamma(K_5) = 2,\) leading to a contradiction. 
		Therefore, \( |A_i| = 2 \) for all \( i \), and so
		\(
		A \cong B_n = \mathbb{Z}_2 \times \cdots \times \mathbb{Z}_2 \, (n \text{ times}),
		\)
		which is a Boolean ring and hence reduced. 
		If \( n \geq 4 \), then by Lemmas~\ref{Zero divisor graph = essential graph}, \ref{pegnl}, and \ref{czg}(i), the genus of \( L(EG(B_n)) \) is at least 3, contradicting our assumption. When \( n = 3 \), Theorem~\ref{pegnl} shows that \( L(EG(B_3)) \) is planar, again a contradiction.  Thus, we conclude that \( n = 2 \), as claimed.

		Now, suppose \(n = 2\). Let us first assume that both \(A_1\) and \(A_2\) are not fields. Then, for each \(i = 1,2\), we have \(|A_i| \geq 4\) and \(|\mathfrak{m}_i^*| \geq 1\) and there exists an element $a_i \in \mathfrak{m}_i^*$ such that $\text{ann}_{A_i}(a_i) = \mathfrak{m}_i.$ Consider the set of vertices $S = \{s_1, s_2, s_3, r_1, r_2, r_3, r_4\}$, where $s_1 = (1,0)$, $s_2 = (u_1,0)$, $s_3 = (a_1,0)$, $r_1 = (0,1)$, $r_2 = (0,u_2)$, $r_3 = (a_1,u_2)$ and $r_4 = (a_1,1)$, and $ u_i \in A_i^\times \setminus \{1\}$. We observe that $s_ir_k = 0$ for $1 \leq i \leq 3$ and $k \in \{1,2\},$  and $s_3r_j = 0$ for  $ j\in \{3,4\}$. Furthermore, $\text{ann}_A(s_lr_m) = \mathfrak{m}_1 \times A_2$ for $l\in \{ 1, 2\}$ and $m \in \{3, 4\}$ is an essential ideal. Therefore, the set \( S \) induces a subgraph in \( EG(A) \) that contains \( K_{3,5} \), implying that \( \overline{\gamma}(L(EG(A))) \geq \overline{\gamma}(L( K_{3,5})) \geq 3 \) by Remark \ref{remark for genus of some  bipartite graphs}, which leads to a contradiction. Hence, at least one of $A_i$ must be a field.
		
		Now, consider the case where \( A_1 \) is a field while \( A_2 \) is not a field with maximal ideal \(\mathfrak{m}_2\). If $|A_1|  \geq 3,$ then $EG(\mathbb Z_3 \times \mathbb Z_4)$ appears as a subgraph in $EG(A_1 \times A_2).$  Therefore, by Lemma \ref{glz34}, we have $\overline{\gamma}(L(EG(A_1 \times A_2))) \geq \overline{\gamma}(L(EG(\mathbb Z_3 \times \mathbb Z_4))) \geq 3$, which leads to a contradiction.  Hence, $|A_1| = 2$, that is, $A_1 \cong \mathbb Z_2$. 
		Next, suppose \( |\mathfrak{m}_2^*| \geq 2 \). Then \( |A_2^\times| \geq 4 \), and there exists \( a \in \mathfrak{m}_2^* \) such that \( \mathrm{ann}_{A_2}(a) = \mathfrak{m}_2 \).  Consider the set of vertices $S'= \{ s_1,s_2,r_1,r_2,r_3,r_4\}$, where $s_1 = (1,0),\, s_2 = (1,a),\, r_1 =(0,1),\, r_2= (0,u_1),\, r_3 =(0,u_2),\, r_4=(0,u_3)$, where  $u_i \, \in \, A_2^\times \setminus \{1\}$. We observe that $s_1r_j= 0$ for $1\leq j \leq 4$. Moreover, for each \( j \), the annihilator \( \text{ann}_{A}(s_2 r_j) = A_1 \times \mathfrak{m}_2 \) is an essential ideal. Hence, the subgraph of \( EG(A) \) induced by \( S' \) contains a copy of \( K_{2,4} \), and so
		\(
		\overline{\gamma}(L(EG(A))) \geq \overline{\gamma}(L(K_{2,4})) = 2,
		\)
		which again contradicts our assumption.
		Therefore, \( |\mathfrak{m}_2^*| = 1 \), implying \( A_2 \cong \mathbb{Z}_4 \) or \( \frac{\mathbb{Z}_2[x]}{\langle x^2 \rangle}\). Thus, the only possible rings in this case are
		\(
		A \cong \mathbb{Z}_2 \times \mathbb{Z}_4 \; \text{or} \; \mathbb{Z}_2 \times \frac{\mathbb{Z}_2[x]}{\langle x^2 \rangle}.
		\)
		
		Suppose that both \(A_1\) and \(A_2\) are fields. In this case, by Lemma~\ref{Zero divisor graph = essential graph}, we have \(EG(A) = \Gamma(A)\).
		Consequently, it follows from  Lemma \ref{czg}(i) that $A\cong \mathbb Z_2 \times \mathbb Z_7$ or $\mathbb F_4 \times \mathbb F_4$.
		
		Conversely, if $A \cong \mathbb{Z}_2 \times \mathbb{Z}_4$ or $\mathbb{Z}_2 \times \frac{\mathbb{Z}_2[x]}{\langle x^2 \rangle}$, it follows form Lemma \ref{glz24} that $\overline{\gamma}(L(EG(A))) = 1$. If $A \cong \mathbb{F}_4 \times \mathbb{F}_4 $\, or\, $ \mathbb{Z}_2 \times \mathbb{Z}_7$, then $A$ is reduced. Hence, the result follows by Lemma \ref{Zero divisor graph = essential graph} and \ref{czg}(ii)(a).
		
	\end{proof}
	\vspace{1pt} 

	Next, we prove that when the crosscap number of the line graph $L(EG(A)$ is two, then $A$ is a reduced ring. Thus, from Lemma \ref{czg}(ii), we get  the  characterization of rings $A$ for which the crosscap number of the graph $L(EG(A))$ is 2.

	\begin{thm}\label{c2leg}
		Let $A$ be a finite non-local commutative ring. Then the crosscap number of the line graph of the essential graph of $A$ is 2 if and only if $A$ is isomorphic to one of the rings: $\mathbb Z_3 \times \mathbb Z_5, \hspace{.2cm} \mathbb Z_2 \times \mathbb Z_2 \times \mathbb Z_3.$
	\end{thm}
	\begin{proof}
		Since \( A \) is a non-local finite ring, we can write \( A \cong A_1 \times A_2 \times \cdots \times A_n \), where each \( A_i \) is a finite local ring and \( n \geq 2 \). Let us assume that \(\overline{\gamma}(L(EG(A))) = 2\).  We claim that $n = 2$ or $3.$ Suppose, for contradiction, that $n\geq 4.$ Consider the vertices 
		$s_1 = (1,0,0,0,\ldots,0)$, 
		$s_2 = (0,1,0,0,\ldots,0)$, 
		$s_3 = (0,0,1,0,\ldots,0)$,
		$s_4 = (1,0,0,1,0,\ldots,0)$, 
		$s_5 = (0,1,1,0,\ldots,0)$, 
		$s_6 = (0,1,0,1,0,\ldots,0)$, 
		$s_7 = (0,0,1,1,0,\ldots,0)$, 
		$s_8 = (1,0,1,0,\ldots,0)$, 
		$s_9 = (1,0,0,1,0, \ldots, 0)$, 
		$s_{10} = (1,1,0,0,\ldots,0)$,
		$s_{11} = (0,1,1,1,0,\ldots,0)$,
		$s_{12} = (1,0,1,1,0,\ldots,0)$ and 
		$s_{13}= (1,1,0,1,0,\ldots,0).$ 
		Note that $s_1s_i = 0$ for $i \in \{2,3,5,6,7,11\},$ $s_2s_j = 0$ for $j \in \{1,3,4,7,8,9,12\}$ and $s_3s_k = 0$ for $ k \in \{1,2,4,6,9,10,13\}.$ Therefore, we find that $\text{deg}(s_1) \geq 6,$ $ \text{deg}(s_2) \geq 7$ and $\text{deg}(s_3) \geq 7.$  Applying Lemma~\ref{lm3}, it follows that \( \overline{\gamma}(L(EG(A))) \geq 3 \), leading to a contradiction. Thus $ n = 2$ or $ 3.$

		Suppose $n = 3.$ We aim to show that each component ring \( A_i \) must be a field. Assume, for instance, that \( A_1 \) is not a field. Then, there exists \( a \in \mathfrak{m}_1^*\) which satiesfies \( \mathrm{ann}_{A_1}(a) = \mathfrak{m}_1.\) Consider the vertices
		$s_1 = (1,0,0,0,\dots,0),\, 
		s_2 = (0,1,0,0,\dots,0),\,
		s_3 = (0,0,1,0,\dots,0),\,
		s_4=(1,1,0,0,\dots,0),\,
		s_5 = (1,0,1,0,\dots,0),\,
		s_6 = (0,1,1,0,\dots,0),\,
		s_7 = (a,1,0,0,\dots,0),\,
		v = (a,0,0,\\0,\dots,0)$. 
		We observe that $\,\text{ann}_A(s_iv)\,$ is an essential ideal for each $1\leq i \leq 7$. Therefore, $\text{deg}(v) \geq 7$ in $EG(A)$, implies that the line graph $L(EG(A))$ contains $K_7$. Therefore, by Proposition \ref{genus and crosscap of Kn and Knm}(i)(b), we have $\overline{\gamma}(L(EG(A))) \geq \overline{\gamma}(K_7) = 3$, which leads to a contradiction. Thus, each  $A_i$ must be a field, meaning that $A$ is a reduced ring. Hence, applying Lemmas \ref{Zero divisor graph = essential graph} and \ref{czg}(ii)(a), it follows that \( A \cong \mathbb{Z}_2 \times \mathbb{Z}_2 \times \mathbb{Z}_3.\)

		Suppose \(n=2\). We aim to show that both \( A_1 \) and \( A_2 \) are fields. Suppose one of $A_1,$ $A_2$ is not a field. Assume, for instance, that \( A_2 \) is not a field. First, we prove that $|A_1| = 2.$ For if,  $|A_1| \geq 3,$ then $EG(A_1 \times A_2)$ contains $ EG(\mathbb Z_3 \times \mathbb Z_4)$ as a subgraph and so,  $\overline{\gamma}(L(EG(A_1 \times A_2))) \geq \overline{\gamma}(L(EG(\mathbb Z_3 \times \mathbb Z_4))) \geq 3$ by Lemma \ref{glz34}, leads to a contradiction. Hence, $A_1 \cong \mathbb Z_2$. Now, we show that $|\mathfrak{m}_2^*| = 1.$ Suppose $|\mathfrak{m}_2^*| \geq 2$. Then we have $|A_2^{\times}| \geq 4$ and there exist $a, b \in \mathfrak{m}_2^*$  with \(ab = 0\) and \(\mathrm{ann}_{A_2}(a) = \mathfrak{m}_2\),  implying \(a^2 = 0\). Consider the vertices $v = (0,a),\, s_1 =(1,a),\, s_2 = (0,b),\,  s_3 =(1,b),\, s_4 = (0,1),\, s_5 = (0, u_1),\, s_6 = (0, u_2),\, s_7 = (0,u_3)$, where $ u_i \in A_2^\times \setminus \{1\}.$ We observe that $vs_i= 0$ for all $1\leq i \leq 3$ and $\text{ann}_A(vs_j) = \mathbb Z_2 \times \mathfrak{m}_2$ for all $ 4\leq j \leq 7$ is an essential ideal. Therefore, $\text{deg}(v) \geq 7$ in $EG(A)$, implies that the line graph $L(EG(A))$ contains $K_7$. Therefore, by Proposition \ref{genus and crosscap of Kn and Knm}(i)(b), we have  $\overline{\gamma}(L(EG(A))) \geq \overline{\gamma}(K_7) = 3$, leading to a contradiction. Hence, $|\mathfrak{m}_2^*| = 1$, implying that  $A_2 $ isomorphic to $\mathbb Z_4$ or $\frac{\mathbb Z_2[x]}{\langle x^2 \rangle}.$ Thus, we obtain  $A \cong \mathbb{Z}_2 \times \mathbb{Z}_4$ or $\mathbb{Z}_2 \times \frac{\mathbb{Z}_2[x]}{\langle x^2 \rangle}.$ But, in this case the crosscap of $L(EG(A))$ is 1 by Lemma \ref{glz24}, which leads to a contradiction.
		Therefore, both \(A_1\) and \(A_2\) must be fields. This implies that the ring \(A\) is reduced. According to Lemma~\ref{Zero divisor graph = essential graph}, it follows that \(EG(A) = \Gamma(A)\). Then, by Lemma~\ref{czg}(ii)(a), in this case the only possible ring is \(A \cong \mathbb{Z}_3 \times \mathbb{Z}_5\).

		Conversly, if \(A \cong \mathbb{Z}_3 \times \mathbb{Z}_5\) or
		$\mathbb Z_2 \times \mathbb Z_2 \times \mathbb Z_3,$ then the ring $A$ is  reduced. Hence, by Lemma \ref{Zero divisor graph = essential graph} and \ref{czg}(ii)(a), $\bar{\gamma}L(EG(A)) = 2$.
	\end{proof}

	\begin{cor}\label{crosscap of essential grh for nonreduced}
		For a non-reduced ring $A$, the crosscap number of $L(EG(A))$ is never 2.
	\end{cor}

	\section*{Acknowledgments}
	The first author gratefully acknowledges Department of Science and Technology(DST), Govt. of India, for the award of the INSPIRE Fellowship [IF220238].

\end{document}